\documentclass[a4paper,11pt,fleqn]{article}
\usepackage{pstricks}
\usepackage{amsmath}
\usepackage{amssymb}
\usepackage{mathrsfs} 
\usepackage{theorem} 
\usepackage{euscript}
\usepackage{exscale,relsize}
\usepackage{graphicx}
\usepackage{srcltx}
\usepackage{xcolor}
\usepackage{charter}
\usepackage[T1]{fontenc}
\usepackage[bitstream-charter]{mathdesign}
\usepackage{url}
\newcommand{\email}[1]{\href{mailto:#1}{\nolinkurl{#1}}}
\PassOptionsToPackage{normalem}{ulem}
\usepackage{ulem}

\renewcommand{\leq}{\ensuremath{\leqslant}}
\renewcommand{\geq}{\ensuremath{\geqslant}}
\newcommand{\minimize}[2]{\ensuremath{\underset{\substack{{#1}}}%
{\text{\rm minimize}}\;\;#2 }}

\newcommand{\argmind}[2]{\ensuremath{\underset{\substack{{#1}}}%
{\text{\rm argmin}}\;\;#2 }}
 
\newcommand{\pair}[2]{\langle{{#1},{#2}}\rangle} 
\newcommand{\Pair}[2]{\big\langle{{#1},{#2}}\big\rangle}

\newcommand{\menge}[2]{\big\{{#1}~\big |~{#2}\big\}}

\newcommand{\IDD}{\ensuremath{\text{\rm int\:dom}f}}
\newcommand{\IDDG}{\ensuremath{\text{\rm int\:dom}g}}
\newcommand{\IDDPS}{\ensuremath{\text{\rm int\:dom}\psi}}

\newcommand{\IDDS}{\ensuremath{\text{\rm int\:dom}f^*}}

\newcommand{\RX}{\ensuremath{\left]-\infty,+\infty\right]}}

\newcommand{\XX}{\ensuremath{{\mathcal X}}}

\newcommand{\YY}{\ensuremath{{\mathcal Y}}}

\newcommand{\emp}{\ensuremath{{\varnothing}}}
\newcommand{\prox}{\ensuremath{\text{\rm prox}}}
\newcommand{\Prox}{\ensuremath{\text{\rm Prox}}}

\newcommand{\Id}{\ensuremath{\operatorname{Id}}\,}

\newcommand{\RR}{\ensuremath{\mathbb{R}}}

\newcommand{\RP}{\ensuremath{\left[0,+\infty\right[}}

\newcommand{\RPP}{\ensuremath{\left]0,+\infty\right[}}
\newcommand{\NN}{\ensuremath{\mathbb N}}

\newcommand{\weakly}{\ensuremath{\:\rightharpoonup\:}}
\newcommand{\exi}{\ensuremath{\exists\,}}
\newcommand{\pinf}{\ensuremath{{+\infty}}}

\newcommand{\dom}{\ensuremath{\text{\rm dom}\,}}

\newcommand{\Cart}{\ensuremath{\raisebox{-0.5mm}{\mbox{\huge{$\times$}}}}}

\newcommand{\inte}{\ensuremath{\text{\rm int}\,}}

\newcommand{\ran}{\ensuremath{\text{\rm ran}\,}}

\newcommand{\gra}{\ensuremath{\text{\rm gra}\,}}

\newcommand{\Argmin}{\ensuremath{\text{\rm Argmin}\,}}

\newcommand{\BP}{\ensuremath{\EuScript P}}

\newcommand{\BF}{\ensuremath{\EuScript F}}
\newcommand{\BFS}{\ensuremath{\EuScript S}}

\newtheorem{theorem}{Theorem}[section]
\newtheorem{lemma}[theorem]{Lemma}
\newtheorem{corollary}[theorem]{Corollary}
\newtheorem{proposition}[theorem]{Proposition}

\theoremstyle{plain}{\theorembodyfont{\rmfamily}%
}
\theoremstyle{plain}{\theorembodyfont{\rmfamily}%
\newtheorem{example}[theorem]{Example}}
\theoremstyle{plain}{\theorembodyfont{\rmfamily}%
\newtheorem{remark}[theorem]{Remark}}
\theoremstyle{plain}{\theorembodyfont{\rmfamily}%
}
\theoremstyle{plain}{\theorembodyfont{\rmfamily}%
\newtheorem{condition}[theorem]{Condition}}
\theoremstyle{plain}{\theorembodyfont{\rmfamily}%
\newtheorem{definition}[theorem]{Definition}}
\theoremstyle{plain}{\theorembodyfont{\rmfamily}
\newtheorem{problem}[theorem]{Problem}}
\theoremstyle{plain}{\theorembodyfont{\rmfamily}
}

\numberwithin{equation}{section}
\definecolor{labelkey}{rgb}{0,0.08,0.45}
\definecolor{refkey}{rgb}{0,0.6,0.0}
\definecolor{Brown}{rgb}{0.45,0.0,0.05}
\definecolor{dgreen}{rgb}{0.00,0.49,0.00}
\definecolor{dblue}{rgb}{0,0.08,0.75}
\RequirePackage[dvips,colorlinks,hyperindex]{hyperref}
\hypersetup{linktocpage=true,citecolor=dblue,linkcolor=dgreen}

\begin{document}

\title{\sffamily\LARGE Forward-Backward Splitting 
with Bregman Distances}

\author{Quang Van Nguyen\\[5mm]
\small
\small Department of Mathematics, Hanoi National University of Education, \\
\small 126 Xuan Thuy Street, Cau Giay dist., Hanoi, Vietnam \\
\small and \\
\small Laboratory for Information Theory and Inference Systems (LIONS), \\
\small \'Ecole Polytechnique F\'ed\'erale de Lausanne, Switzerland.\\
\small \email{quangnv@hnue.edu.vn, quang.nguyen@epfl.ch}
}
\date{~}
\maketitle
\setcounter{page}{1}

\vskip 8mm

\begin{abstract}
We propose a forward-backward splitting algorithm 
based on Bregman distances for composite minimization 
problems in general reflexive Banach spaces.
The convergence is established using the notion of 
variable quasi-Bregman monotone sequences. Various examples are
discussed, including some in Euclidean spaces, where new algorithms
are obtained. 
\end{abstract}

{\bfseries Key words.}
Banach space,
Bregman distance, 
forward-backward algorithm,
Legendre function, 
multivariate minimization, 
variable quasi-Bregman monotonicity.

{\bfseries Mathematics Subject Classifications (2010)} 
90C25

\newpage

\section{Introduction}
\noindent
We consider the following composite convex 
minimization problem.

\begin{problem}
\label{IIIpb:1}
Let $\XX$ and $\YY$ be reflexive real Banach spaces, 
let $\varphi\in\Gamma_0(\XX)$, 
let $\psi\in\Gamma_0(\YY)$ be G\^ateaux differentiable on 
$\IDDPS\neq\emp$, and let $L\colon\XX\to\YY$ be 
a bounded linear operator. The problem is to
\begin{equation}
\label{IIIeq:1}
\minimize{x\in \XX}{\varphi(x)+\psi(Lx)}.
\end{equation}
The set of solutions to \eqref{IIIeq:1} is denoted by $\BFS$.
\end{problem}

A particular instance of \eqref{IIIeq:1} is 
when $\psi=D^g(\cdot,r)$, where 
$g\in\Gamma_0(\YY)$ is G\^ateaux differentiable on 
$\IDDG\ni r$, i.e.,
\begin{equation}
\label{IIIe:56}
\minimize{x\in \XX}{\varphi(x)+D^g(Lx,r)}.
\end{equation}
This model provides a framework for many problems arising in applied 
mathematics. For instance, when $\XX$ and $\YY$ are Euclidean 
spaces and $g$ is Boltzmann-Shannon entropy, 
it captures many problems in 
information theory and signal recovery \cite{IIIBBDV09}. 
Besides, the nearness matrix problem \cite{IIIDHTR07} and the 
log-determinant minimization problem \cite{IIICPW12} 
can be also regarded as special cases of \eqref{IIIe:56}.

An objective is constructing effective 
splitting methods, i.e, the methods that activate 
each function in the model separately, 
to solve Problem~\ref{IIIpb:1} 
(see \cite{IIICW05} for more discussions). 
It was shown in \cite{IIICW05} that if $\XX$ and $\YY$ are 
Hilbert spaces and if $\psi$ possess a $\beta^{-1}$-Lipschitz 
continuous gradient for some $\beta\in\RPP$, 
then Problem~\ref{IIIpb:1} can be solved by 
the standard forward-backward algorithm
\begin{equation}
\label{IIeq:1b}
(\forall n\in\NN)\quad x_{n+1}=\prox_{\gamma_n\varphi}
\big(x_n-\gamma_nL^*(\nabla\psi(Lx_n))\big),
\quad\text{where}\quad 0<\gamma_n<2\beta.
\end{equation}
Here, $(\prox_{\gamma_n\varphi})_{n\in\NN}$ 
are the Moreau proximity operators \cite{IIIMor62b}. 
However, many problems in applications do not conform to these 
hypotheses, for example when $\XX$ and $\YY$ are Euclidean spaces 
and $\psi$ is Boltzmann-Shannon entropy. This 
type of functions appears in many problems in image and signal 
processing, in statistics, and in machine learning 
\cite{IIIBBM07,IIIByr93,IIICPW12,IIIKW01,IIIHMC02,IIIMC01}. 
Another difficulty in the 
implementation of \eqref{IIeq:1b} is that the operators 
$(\prox_{\gamma_n\varphi})_{n\in\NN}$ are not always easy 
to evaluate. The objective of the present paper is to propose 
a version of the forward-backward splitting algorithm to solve 
Problem~\ref{IIIpb:1}, which is so far limited to Hilbert spaces, 
in the general framework of reflexive real Banach spaces. 
This algorithm, which employs Bregman distance-based proximity operators, 
provides new algorithms in the framework 
of Euclidean spaces, which are, in some instances, more favorable 
than the standard forward-backward splitting algorithm. 
This framework can be applied in the case when $\psi$ 
is not everywhere differentiable and in some instances, 
it requires less efforts in the computation of proximity operators 
than the classical framework. This paper revolves around 
the following definitions.

\begin{definition}{\rm\cite{IIIBBC01,IIIBBC03}}
Let $\XX$ be a reflexive real Banach space, 
let $\XX^*$ be the topological dual space of $\XX$, 
let $\pair{\cdot}{\cdot}$ be the duality pairing between 
$\XX$ and $\XX^*$, let $f\colon\XX\to\RX$ be a 
lower semicontinuous convex function that is 
G\^ateaux differentiable on $\IDD\neq\emp$, 
let $f^*\colon\XX^*\to\RX\colon
x^*\mapsto\sup_{x\in\XX}(\pair{x}{x^*}-f(x))$ be 
conjugate of $f$, and let
\begin{equation}
\label{IIIe:subdiff}
\partial f\colon\XX\to 2^{\XX^*}\colon x\mapsto
\menge{x^*\in\XX^*}{(\forall y\in\XX)\,
\pair{y-x}{x^*}+f(x)\leq f(y)},  
\end{equation}
be the Moreau subdifferential of $f$. 
The \emph{Bregman distance} associated with $f$ is 
\begin{equation}
\label{IIIe:Bdist}
\begin{aligned}
D^f\colon\XX\times\XX&\to\,\left[0,\pinf\right]\\
(x,y)&\mapsto 
\begin{cases}
f(x)-f(y)-\Pair{x-y}{\nabla f(y)},&\text{if}\;\;y\in\IDD;\\
\pinf,&\text{otherwise}.
\end{cases}
\end{aligned}
\end{equation}
In addition, $f$ is a \emph{Legendre function} if it is 
\emph{essentially smooth} in the sense that $\partial f$ is 
both locally bounded and single-valued on its
domain, and \emph{essentially strictly convex} in the sense
that $\partial f^*$ is locally bounded on its domain and 
$f$ is strictly convex on every convex subset of $\dom\partial f$.
Let $C$ be a closed convex subset of $\XX$ such that
$C\cap\IDD\neq\emp$. The \emph{Bregman projector} onto 
$C$ induced by $f$ is 
\begin{equation}
\label{IIIe:2001}
\begin{aligned}
P^f_C\colon\IDD&\to C\cap\IDD\\
y&\mapsto\underset{x\in C}{\text{argmin}}\,D^f(x,y),
\end{aligned}
\end{equation}
and the \emph{$D^f$-distance} 
to $C$ is the function
\begin{equation}
\begin{aligned}
D^f_C\colon\XX&\to\left[0,+\infty\right]\\
y&\mapsto\inf D^f(C,y).
\end{aligned}
\end{equation}
\end{definition}

The paper is organized as follows. In Section~\ref{IIIsec:tech}, 
we provide some preliminary results. We present 
the forward-backward splitting algorithm in 
reflexive Banach spaces in Section~\ref{IIIsec:FB}. 
Section~\ref{IIIsect:mul} is devoted to an application 
of our result to multivariate minimization problem 
together with examples.

\noindent {\bf Notation and background.} 
The norm of a Banach space is denoted by $\|\cdot\|$. 
The symbols $\weakly$ and $\to$ represent respectively weak and 
strong convergence. The set of weak sequential cluster points 
of a sequence $(x_n)_{n\in\NN}$ is denoted by 
$\mathfrak{W}(x_n)_{n\in\NN}$. 
Let $M\colon\XX\to 2^{\XX^*}$. 
The domain of $M$ is $\dom M=\menge{x\in\XX}{Mx\neq\emp}$ 
and the range of $M$ is 
$\ran M=\menge{x^*\in\XX^*}{(\exi x\in\XX)\,x^*\in Mx}$. 
Let $f\colon\XX\to\RX$. Then $f$ is cofinite if 
$\dom f^*=\XX^*$, is coercive if 
$\lim_{\|x\|\to+\infty}f(x)=+\infty$, is supercoercive if 
$\lim_{\|x\|\to+\infty}f(x)/\|x\|=+\infty$, and is uniformly 
convex at $x\in\dom f$ if there exists an increasing 
function $\phi\colon\left[0,+\infty\right[
\to\left[0,+\infty\right]$ that vanishes only at $0$ such that
\begin{equation}
(\forall y\in\dom f)
(\forall \alpha\in\left]0,1\right[)\quad
f(\alpha x+(1-\alpha)y)+\alpha(1-\alpha)\phi(\|x-y\|)
\leq\alpha f(x)+(1-\alpha)f(y).
\end{equation}
Denote by $\Gamma_0(\XX)$ the class of all 
lower semicontinuous convex functions $f\colon\XX\to\RX$ 
such that $\dom f=\menge{x\in\XX}{f(x)<+\infty}\neq\emp$. 
Let $f\in\Gamma_0(\XX)$. The set of global minimizers of 
$f$ is denoted by $\Argmin f$. 
Finally, $\ell_{+}^1(\NN)$ is the set of all 
summable sequences in $\left[0,+\infty\right[$.

\section{Preminarily results}
\label{IIIsec:tech}
\noindent
First, we recall the following definitions and results.
\begin{definition}{\rm\cite{IIIQ}}
Let $\XX$ be a reflexive real Banach space 
and let $f\in\Gamma_0(\XX)$ be G\^ateaux differentiable on 
$\IDD\neq\emp$. Then
\begin{equation}
\BF(f)=\menge{g\in\Gamma_0(\XX)}{g\;\text{is G\^ateaux 
differentiable on}\;\IDDG=\IDD}.
\end{equation}
Moreover, if $g_1$ and $g_2$ are in $\BF(f)$, then
\begin{equation}
g_1\succcurlyeq g_2\quad\Leftrightarrow\quad
(\forall x\in\dom f)(\forall y\in\IDD)
\quad D^{g_1}(x,y)\geq D^{g_2}(x,y).
\end{equation}
For every $\alpha\in\RP$, set
\begin{equation}
\BP_{\alpha}(f)=\menge{g\in\BF(f)}{g\succcurlyeq\alpha f}.
\end{equation}
\end{definition}

\begin{definition}{\rm\cite{IIIQ}}
\label{IIIdef:1}
Let $\XX$ be a reflexive real Banach space, 
let $f\in\Gamma_0(\XX)$ be G\^ateaux differentiable on 
$\IDD\neq\emp$, let $(f_n)_{n\in\NN}$ 
be in $\BF(f)$, let $(x_n)_{n\in\NN}\in(\IDD)^{\NN}$, 
and let $C\subset\XX$ be such that 
$C\cap\dom f\neq\emp$. Then $(x_n)_{n\in\mathbb{N}}$ is:  
\begin{enumerate} 
\item
\emph{quasi-Bregman monotone} with respect to $C$ relative to
$(f_n)_{n\in\mathbb{N}}$ if 
\begin{multline}
\label{IIIe:qbm}
(\exists(\eta_n)_{n\in\mathbb{N}}\in\ell_+^1(\mathbb{N}))
(\forall x\in C\cap\dom f)
(\exists(\varepsilon_n)_{n\in\mathbb{N}}\in
\ell_+^1(\mathbb{N}))(\forall n\in\mathbb{N})\\ 
D^{f_{n+1}}(x,x_{n+1})\leq
(1+\eta_n)D^{f_n}(x,x_n)+\varepsilon_n;
\end{multline}
\item 
\emph{stationarily quasi-Bregman monotone} with respect to $C$
relative to $(f_n)_{n\in\mathbb{N}}$ if
\begin{multline}
\label{IIIe:sqbm}
(\exists(\varepsilon_n)_{n\in\mathbb{N}}\in\ell_+^1(\mathbb{N}))
(\exists(\eta_n)_{n\in\mathbb{N}}\in\ell_+^1(\mathbb{N}))(\forall
x\in C\cap\dom f)(\forall n\in\mathbb{N})\\
D^{f_{n+1}}(x,x_{n+1})\leq (1+\eta_n)D^{f_n}(x,x_n)+\varepsilon_n.
\end{multline}
\end{enumerate}
\end{definition}

\begin{condition}
{\rm\cite[Condition~4.4]{IIIBBC03}}
\label{IIIcd:2}
Let $\XX$ be a reflexive real Banach space
and let $f\in\Gamma_0(\XX)$ be G\^ateaux differentiable on 
$\IDD\neq\emp$. For every bounded sequences 
$(x_n)_{n\in\NN}$ and $(y_n)_{n\in\NN}$ in $\IDD$,
\begin{equation}
\label{IIIe:2}
D^f(x_n,y_n)\to 0\quad\Rightarrow\quad x_n-y_n\to 0.
\end{equation}
\end{condition}

\begin{proposition}{\rm\cite{IIIQ}}
\label{IIIpp:qb1}
Let $\XX$ be a reflexive real Banach space, 
let $f\in\Gamma_0(\XX)$ be G\^ateaux differentiable on 
$\IDD\neq\emp$, let $\alpha\in\RPP$, 
let $(f_n)_{n\in\NN}$ be in $\BP_{\alpha}(f)$, let 
$(x_n)_{n\in\NN}\in(\IDD)^{\NN}$, let 
$C\subset\XX$ be such that $C\cap\IDD\neq\emp$, 
and let $x\in C\cap\IDD$. Suppose that 
$(x_n)_{n\in\mathbb{N}}$ is quasi-Bregman monotone with 
respect to $C$ relative to $(f_n)_{n\in\mathbb{N}}$. 
Then the following hold.
\begin{enumerate}
\item
\label{IIIpp:qb1i}
$(D^{f_n}(x,x_n))_{n\in\NN}$ converges.
\item
\label{IIIpp:qb1ii}
Suppose that $D^f(x,\cdot)$ is coercive. Then 
$(x_n)_{n\in\mathbb{N}}$ is bounded.
\end{enumerate}
\end{proposition}

\begin{proposition}{\rm\cite{IIIQ}}
\label{IIIpp:qb2}
Let $\XX$ be a reflexive real Banach space, 
let $f\in\Gamma_0(\XX)$ be G\^ateaux differentiable 
on $\IDD\neq\emp$, let $(x_n)_{n\in\NN}\in(\IDD)^{\NN}$, 
let $C\subset\XX$ be such that $C\cap\IDD\neq\emp$, 
let $(\eta_n)_{n\in\NN}\in\ell_{+}^1(\NN)$, 
let $\alpha\in\RPP$, and let $(f_n)_{n\in\NN}$ in 
$\BP_{\alpha}(f)$ be such that $(\forall n\in\NN)$ 
$(1+\eta_n)f_n\succcurlyeq f_{n+1}$. Suppose that 
$(x_n)_{n\in\mathbb{N}}$ is quasi-Bregman 
monotone with respect to $C$ relative to 
$(f_n)_{n\in\mathbb{N}}$, that there exists $g\in\BF(f)$ 
such that for every $n\in\NN$, $g\succcurlyeq f_n$, 
and that, for every 
$y_1\in\XX$ and every $y_2\in\XX$,
\begin{equation}
\label{IIIe:4p}
\begin{cases}
y_1\in\mathfrak{W}(x_n)_{n\in\NN}\cap C\\
y_2\in\mathfrak{W}(x_n)_{n\in\NN}\cap C\\
\big(\Pair{y_1-y_2}{\nabla f_n(x_n)}\big)_{n\in\NN}
\quad\text{converges}
\end{cases}
\Rightarrow\quad y_1=y_2.
\end{equation}
Moreover, suppose that $(\forall x\in\IDD)$ 
$D^f(x,\cdot)$ is coercive. Then $(x_n)_{n\in\NN}$ converges 
weakly to a point in $C\cap\IDD$ if and only if 
$\mathfrak{W}(x_n)_{n\in\NN}\subset C\cap\IDD$.
\end{proposition}

\begin{proposition}{\rm\cite{IIIQ}}
\label{IIIpp:qb3} 
Let $\XX$ be a reflexive real Banach space, 
let $f\in\Gamma_0(\XX)$ be a Legendre function, 
let $\alpha\in\RPP$, 
let $(f_n)_{n\in\NN}$ be in $\BP_{\alpha}(f)$, 
let $(x_n)_{n\in\NN}\in(\IDD)^{\NN}$, and let 
$C$ be a closed convex subset of $\XX$ such that 
$C\cap\IDD\neq\emp$. Suppose that 
$(x_n)_{n\in\mathbb{N}}$ is stationarily 
quasi-Bregman monotone with respect to $C$ relative to 
$(f_n)_{n\in\mathbb{N}}$, that $f$ satisfies 
Condition~\ref{IIIcd:2}, and that 
$(\forall x\in\IDD)$ $D^f(x,\cdot)$ is coercive. 
In addition, suppose that there exists $\beta\in\RPP$ 
such that $(\forall n\in\NN)$ 
$\beta f\succcurlyeq f_n$. 
Then $(x_n)_{n\in\NN}$ converges 
strongly to a point in $C\cap\overline{\dom}f$ 
if and only if $\varliminf D^f_C(x_n)=0$.
\end{proposition}

We discuss some basic properties of 
a type of Bregman distance-based proximity operators 
in the following proposition.

\begin{proposition}
\label{IIIl:fprox}
Let $\XX$ be a reflexive real Banach space, 
let $f\in\Gamma_0(\XX)$ be G\^ateaux differentiable on 
$\IDD\neq\emp$, let $\varphi\in\Gamma_0(\XX)$, 
and let
\begin{equation}
\begin{aligned}
\label{IIIe:fprox}
\Prox_{\varphi}^f\colon\XX^*&\to 2^{\XX}\\ 
x^*&\mapsto\menge{x\in\XX}{\varphi(x)+f(x)-\pair{x}{x^*}
=\min\big(\varphi+f-x^*)(\XX)<+\infty}
\end{aligned}
\end{equation}
be \emph{$f$-proximity operator} of $\varphi$. 
Then the following hold.
\begin{enumerate}
\item
\label{IIIl:fproxi} 
$\ran\Prox_{\varphi}^f\subset\dom f\cap \dom\varphi$ and 
$\Prox_{\varphi}^f=(\partial(f+\varphi))^{-1}$.
\item
\label{IIIl:fproxii}
Suppose that $\dom\varphi\cap\IDD\neq\emp$ 
and that $\dom\partial f\cap\dom\partial\varphi\subset
\IDD$. Then the following hold.
\begin{enumerate}
\item
\label{IIIl:fproxiia}
$\ran\Prox_{\varphi}^f\subset\IDD$ 
and $\Prox_{\varphi}^f=(\nabla f+\partial\varphi)^{-1}$.
\item
\label{IIIl:fproxiib} 
$\inte(\dom f^*+\dom\varphi^*)\subset\dom\Prox_{\varphi}^f$.
\item
\label{IIIl:fproxiic}
Suppose that $f|_{\IDD}$ is strictly convex. 
Then $\Prox_{\varphi}^f$ is single-valued on its domain.
\end{enumerate}
\end{enumerate}
\end{proposition}

\begin{proof}
Let us fix $x^*\in\XX^*$ and define 
$f_{x^*}\colon\XX\to\RX\colon x\mapsto f(x)-\pair{x}{x^*}
+f^*(x^*)$. Then $\dom f_{x^*}=\dom f$ and $\varphi+f_{x^*}
\in\Gamma_0(\XX)$. Moreover, $\partial(\varphi+f_{x^*})
=\partial(\varphi+f)-x^*$.

\ref{IIIl:fproxi}: 
By definition, $\ran\Prox_{\varphi}^f\subset
\dom f\cap\dom\varphi$. Suppose that 
$\dom f\cap\dom\varphi\neq\emp$ and let 
$x\in\dom f\cap\dom\varphi$. Then
\begin{align}
x\in\Prox_{\varphi}^fx^*&\Leftrightarrow 0\in\partial
(\varphi+f_{x^*})(x)\nonumber\\
&\Leftrightarrow 0\in\partial(\varphi+f)(x)-x^*\nonumber\\
&\Leftrightarrow x^*\in\partial(\varphi+f)(x)\nonumber\\
&\Leftrightarrow x\in
\big(\partial(\varphi+f)\big)^{-1}(x^*).
\end{align}

\ref{IIIl:fproxii}: 
Suppose that $x^*\in\inte(\dom f^*+\dom\varphi^*)$. Since 
$\dom\varphi\cap\IDD\neq\emp$, it follows from 
\cite[Theorem~1.1]{IIIAB86} and 
\cite[Theorem~2.1.3(ix)]{IIIZa02_B} that 
\begin{equation}
\label{IIe:30g}
x^*\in\inte(\dom f^*+\dom\varphi^*)=\inte\dom(f+\varphi)^*.
\end{equation}

\ref{IIIl:fproxiia}: 
Since $\dom\varphi\cap\IDD\neq\emp$, $\partial(\varphi+f)
=\partial\varphi+\partial f$ by \cite[Corollary~2.1]{IIIAB86}, 
and hence \ref{IIIl:fproxi} yields
\begin{equation}
\ran\Prox_{\varphi}^f=\dom\partial(f+\varphi)=
\dom(\partial f+\partial\varphi)=\dom\partial f\cap
\dom\partial\varphi\subset\IDD.
\end{equation}
In turn, $\ran\Prox_{\varphi}^f\subset\dom\varphi\cap\IDD$. 
We now prove that 
$\Prox_{\varphi}^f=(\nabla f+\partial\varphi)^{-1}$. 
Note that $\dom(\nabla f+\partial\varphi)
\subset\dom\varphi\cap\IDD$. Let 
$x\in\dom\varphi\cap\IDD$. Then 
$\partial(f+\varphi)(x)=\partial f(x)+\partial\varphi(x)
=\nabla f(x)+\partial\varphi(x)$ and therefore,
\begin{equation}
x\in\Prox_{\varphi}^fx^*\Leftrightarrow x^*\in
\partial(f+\varphi)(x)=\nabla f(x)+\partial
\varphi(x)\Leftrightarrow x\in(\nabla f+\partial
\varphi)^{-1}(x^*).
\end{equation}

\ref{IIIl:fproxiib}: 
We derive from \eqref{IIe:30g} and 
\cite[Fact~3.1]{IIIBBC01} that $\varphi+f_{x^*}$ is coercive. 
Hence, by \cite[Theorem~2.5.1]{IIIZa02_B}, $\varphi+f_{x^*}$ 
admits at least one minimizer, i.e., 
$x^*\in\dom\Prox_{\varphi}^f$.

\ref{IIIl:fproxiic}: 
Since $f|_{\IDD}$ is strictly convex, so is 
$(\varphi+f_{x^*})|_{\IDD}$ 
and thus, in view of \ref{IIIl:fproxiib}, 
$\varphi+f_{x^*}$ admits a unique minimizer on $\IDD$. 
However, since
\begin{equation}
\Argmin(\varphi+f_{x^*})=\ran\Prox_{\varphi}^f\subset\IDD,
\end{equation}
it follows that $\varphi+f_{x^*}$ admits a unique minimizer 
and that $\Prox_{\varphi}^f$ is therefore single-valued.
\end{proof}

\begin{proposition}
\label{IIIl:fproxprod}
Let $m$ be a strictly positive integer, let 
$(\XX_i)_{1\leq i\leq m}$ be reflexive real 
Banach spaces, and let $\XX$ be the vector 
product space $\Cart_{\!\!i=1}^{\!\!m}\XX_i$ equipped 
with the norm $x=(x_i)_{1\leq i\leq m}
\mapsto\sqrt{\sum_{i=1}^m\|x_i\|^2}$. 
For every $i\in\{1,\ldots,m\}$, 
let $f_i\in\Gamma_0(\XX_i)$ be a Legendre function 
and let $\varphi_i\in\Gamma_0(\XX_i)$ be such that 
$\dom\varphi_i\cap\IDD_i\neq\emp$. 
Set $f\colon\XX\to\RX
\colon x\mapsto\sum_{i=1}^mf_i(x_i)$ and 
$\varphi\colon\XX\to\RX\colon 
x\mapsto\sum_{i=1}^m\varphi_i(x_i)$. Then
\begin{equation}
\Big(\forall x^*=(x_i^*)_{1\leq i\leq m}
\in\Cart_{\!\!i=1}^{\!\!m}\inte(\dom f_i^*+\dom\varphi_i^*)\Big)
\quad\Prox_{\varphi}^fx^*
=\big(\Prox_{\varphi_i}^{f_i}x_i^*\big)_{1\leq i\leq m}.
\end{equation}
\end{proposition}
\begin{proof}
First, we observe that $\XX^*$ is the vector 
product space $\Cart_{\!\!i=1}^{\!\!m}\XX_i^*$ equipped 
with the norm $x^*=(x_i^*)_{1\leq i\leq m}
\mapsto\sqrt{\sum_{i=1}^m\|x_i^*\|^2}$. 
Next, we derive from the definition of $f$ 
that $\dom f=\Cart_{\!\!i=1}^{\!\!m}\dom f_i$ 
and that
\begin{equation}
\partial f\colon\XX\to 
2^{\XX^*}\colon (x_i)_{1\leq i\leq m}
\mapsto\underset{i=1}{\overset{m}{\Cart}}\partial f_i(x_i).
\end{equation}
Thus, $\partial f$ is single-valued on 
\begin{equation}
\dom\partial f
=\underset{i=1}{\overset{m}{\Cart}}\dom\partial f_i
=\underset{i=1}{\overset{m}{\Cart}}\IDD_i
=\inte\Big(\underset{i=1}{\overset{m}{\Cart}}\dom f_i\Big)
=\IDD.
\end{equation}
Likewise, since 
\begin{equation}
f^*\colon\XX^*\to\RX\colon 
(x_i^*)_{1\leq i\leq m}\mapsto\sum_{i=1}^mf_i^*(x_i^*), 
\end{equation}
we deduce that $\partial f^*$ is single-valued on 
$\dom\partial f^*=\IDDS$. 
Consequently, \cite[Theorems~5.4 and~5.6]{IIIBBC01} assert that 
\begin{equation}
\label{IIIe:leg}
f\;\;\text{is a Legendre function.}
\end{equation}
In addition,
\begin{equation}
\label{IIIe:inter}
\dom\varphi\cap\IDD
=\Big(\underset{i=1}{\overset{m}{\Cart}}\dom\varphi_i\Big)
\cap\Big(\underset{i=1}{\overset{m}{\Cart}}\IDD_i\Big)
=\underset{i=1}{\overset{m}{\Cart}}(\dom\varphi_i
\cap\IDD_i)\neq\emp.
\end{equation}
Hence, 
Proposition~\ref{IIIl:fprox}\ref{IIIl:fproxiib}\&\ref{IIIl:fproxiic} 
assert that $\inte(\dom f^*+\dom\varphi^*)
\subset\dom\Prox_{\varphi}^f$ and $\Prox_{\varphi}^f$ 
is single-valued on its domain. Now set 
$x=\Prox^f_{\varphi}x^*$ and 
$q=(\Prox^{f_i}_{\varphi_i}x_i^*)_{1\leq i\leq m}$. 
We derive from Proposition~\ref{IIIl:fprox}\ref{IIIl:fproxiia} that
\begin{equation}
x=\Prox_{\varphi}^f
x^*\Leftrightarrow x
=(\nabla f+\partial\varphi)^{-1}(x^*)
\Leftrightarrow x^*-\nabla f(x)\in\partial\varphi(x).
\end{equation}
Consequently, by invoking \eqref{IIIe:subdiff}, we get
\begin{equation}
\label{IIIe:31i}
(\forall z\in\dom\varphi)\quad
\Pair{z-x}{x^*-\nabla f(x)}+\varphi(x)\leq\varphi(z).  
\end{equation}
Upon setting $z=q$ in \eqref{IIIe:31i}, we obtain
\begin{equation}
\label{IIIe:32i}
\Pair{q-x}{x^*-\nabla f(x)}+\varphi(x)\leq\varphi(q).
\end{equation}
For every $i\in\{1,\ldots,m\}$, let us set 
$q_i=\Prox_{\varphi_i}^{f_i}x_i^*$. The same characterization 
as in \eqref{IIIe:31i} yields
\begin{equation}
\label{IIIe:33i}
(\forall i\in\{1,\ldots,m\})(\forall z_i\in\dom\varphi_i)
\quad\Pair{z_i-q_i}{x_i^*-\nabla f_i(q_i)}+\varphi_i(q_i)
\leq\varphi_i(z_i).
\end{equation}
By summing these inequalities over $i\in\{1,\ldots,m\}$, 
we obtain
\begin{equation}
\label{IIIe:34i}
(\forall z\in\dom\varphi)\quad
\Pair{z-q}{x^*-\nabla f(q)}+\varphi(q)\leq\varphi(z).
\end{equation}
Upon setting $z=x$ in \eqref{IIIe:34i}, 
we get
\begin{equation}
\label{IIIe:35i}
\Pair{x-q}{\nabla f(x)-\nabla f(q)}+\varphi(q)
\leq\varphi(x).
\end{equation}
Adding \eqref{IIIe:32i} and \eqref{IIIe:35i} yields
\begin{equation}
\label{IIIe:36i}
\Pair{x-q}{\nabla f(x)-\nabla f(q)}\leq 0.
\end{equation}
Now suppose that $x\neq q$. Since $f|_{\IDD}$ is strictly 
convex, it follows from \cite[Theorem~2.4.4(ii)]{IIIZa02_B} 
that $\nabla f$ is strictly monotone, i.e.,
\begin{equation}
\Pair{x-q}{\nabla f(x)-\nabla f(q)}>0,
\end{equation}
and we reach a contradiction.
\end{proof}

In Hilbert spaces, the operator defined in \eqref{IIIe:fprox} 
reduces to the Moreau's usual proximity operator $\prox_{\varphi}$
\cite{IIIMor62b} if $f=\|\cdot\|^2/2$.  
We provide illustrations of instances in the standard Euclidean 
space $\RR^m$ in which $\Prox_{\varphi}^f$ is easier to 
evaluate than $\prox_{\varphi}$.

\begin{example}
\label{IIIex:1a}
Let $\gamma\in\RPP$, let $\phi\in\Gamma_0(\RR)$ be such that 
$\dom\phi\cap\RPP\neq\emp$, and let $\vartheta$ be 
Boltzmann-Shannon entropy, i.e.,
\begin{equation}
\vartheta\colon\xi\mapsto
\begin{cases}
\xi\ln\xi-\xi,&\text{if}\;\;\xi\in\RPP;\\
0,&\text{if}\;\;\xi=0;\\
\pinf,&\text{otherwise}.
\end{cases}
\end{equation}
Set $\varphi\colon(\xi_i)_{1\leq i\leq m}\mapsto
\sum_{i=1}^m\phi(\xi_i)$ and $f\colon(\xi_i)_{1\leq i\leq m}
\mapsto\sum_{i=1}^m\vartheta(\xi_i)$. Note that $f$ is a 
supercoercive Legendre function \cite[Sections~5 and 6]{IIIBB97}, 
and hence, Proposition~\ref{IIIl:fprox}\ref{IIIl:fproxiib} 
asserts that $\dom\Prox_{\varphi}^f=\RR^m$. 
Let $(\xi_i)_{1\leq i\leq m}\in\RR^m$, set 
$(\eta_i)_{1\leq i\leq m}
=\Prox_{\gamma\varphi}^f(\xi_i)_{1\leq i\leq m}$, 
let $W$ be the Lambert function \cite{IIIlamb96}, 
i.e., the inverse of $\xi\mapsto\xi e^{\xi}$ on $\RP$, 
and let $i\in\{1,\ldots,m\}$. 
Then $\eta_i$ can be computed as follows.
\begin{enumerate}
\item
\label{IIIex:1ii}
Let $\omega\in\RR$ and suppose that 
\begin{equation}
\phi\colon\xi\mapsto
\begin{cases}
\xi\ln\xi-\omega\xi,&\text{if}\;\;\xi\in\RPP;\\
0,&\text{if}\;\;\xi=0;\\
+\infty,&\text{otherwise}.
\end{cases}
\end{equation}
Then $\eta_i=e^{(\xi_i+\omega-1)/(\gamma+1)}$.
\item 
\label{IIIex:1iii}
Let $p\in\left[1,+\infty\right[$ 
and suppose that either $\phi=|\cdot|^p/p$ or
\begin{equation}
\phi\colon\xi\mapsto
\begin{cases}
\xi^p/p,&\text{if}\;\;\xi\in\RP;\\
+\infty,&\text{otherwise}.
\end{cases}
\end{equation}
Then 
\begin{equation}
\eta_i=
\begin{cases}
\left(\dfrac{W(\gamma(p-1)e^{(p-1)\xi_i})}{\gamma(p-1)}
\right)^{\frac{1}{p-1}},&\text{if}\;\;p\in
\left]1,+\infty\right[;\\[4mm]
e^{\xi_i-\gamma},&\text{if}\;\;p=1.
\end{cases}
\end{equation}
\item
\label{IIIex:1v}
Let $p\in\left[1,+\infty\right[$ and suppose that
\begin{equation}
\phi\colon\xi\mapsto
\begin{cases}
\xi^{-p}/p,&\text{if}\;\;\xi\in\RPP;\\
+\infty,&\text{otherwise}.
\end{cases}
\end{equation}
Then 
\begin{equation}
\eta_i=\left(\frac{W(\gamma(p+1)e^{-(p+1)\xi_i})}{\gamma(p+1)}
\right)^{\frac{-1}{p+1}}.
\end{equation}
\item
\label{IIIex:1vi}
Let $p\in\left]0,1\right[$ and suppose that
\begin{equation}
\phi\colon\xi\mapsto
\begin{cases}
-\xi^p/p,&\text{if}\;\;\xi\in\RP;\\
+\infty,&\text{otherwise}.
\end{cases}
\end{equation} 
Then 
\begin{equation}
\eta_i=\Bigg(\frac{W(\gamma(1-p)e^{(p-1)\xi_i})}{\gamma(1-p)}
\Bigg)^{\frac{1}{p-1}}.
\end{equation}
\end{enumerate}
\end{example}

\begin{example}
\label{IIIex:2a}
Let $\phi\in\Gamma_0(\RR)$ be such that 
$\dom\phi\cap\left]0,1\right[\neq\emp$ 
and let $\vartheta$ be Fermi-Dirac entropy, i.e.,
\begin{equation}
\vartheta\colon\xi\mapsto
\begin{cases}
\xi\ln\xi-(1-\xi)\ln(1-\xi),&\text{if}\;\;\xi\in\left]0,1\right[;\\
0&\text{if}\;\;\xi\in\{0,1\};\\
\pinf,&\text{otherwise}.
\end{cases}
\end{equation}
Set $\varphi\colon(\xi_i)_{1\leq i\leq m}
\mapsto\sum_{i=1}^m\phi(\xi_i)$ 
and $f\colon(\xi_i)_{1\leq i\leq m}
\mapsto\sum_{i=1}^m\vartheta(\xi_i)$. 
Note that $f$ is a cofinite Legendre function 
\cite[Sections~5 and 6]{IIIBB97}, and hence 
Proposition~\ref{IIIl:fprox}\ref{IIIl:fproxiib} 
asserts that $\dom\Prox_{\varphi}^f=\RR^m$. 
Let $(\xi_i)_{1\leq i\leq m}\in\RR^m$, 
set $(\eta_i)_{1\leq i\leq m}
=\Prox_{\varphi}^f(\xi_i)_{1\leq i\leq m}$, 
and let $i\in\{1,\ldots,m\}$. 
Then $\eta_i$ can be computed as follows.
\begin{enumerate}
\item
\label{IIIex:2i}
Let $\omega\in\RR$ and suppose that
\begin{equation}
\phi\colon\xi\mapsto
\begin{cases}
\xi\ln\xi-\omega\xi,&\text{if}\;\;\xi\in\RPP;\\
0,&\text{if}\;\;\xi=0;\\
\pinf,&\text{otherwise}.
\end{cases}
\end{equation}
Then $\eta_i=-e^{\xi_i+\omega-1}/2+\sqrt{e^{2(\xi_i+\omega-1)}/4
+e^{\xi_i+\omega-1}}$.
\item
\label{IIIex:2ii}
Suppose that
\begin{equation}
\phi\colon\xi\mapsto
\begin{cases}
(1-\xi)\ln(1-\xi)+\xi,&\text{if}\;\;\xi\in\left]-\infty,1\right[;\\
1,&\text{if}\;\;\xi=1;\\
\pinf,&\text{otherwise}.
\end{cases}
\end{equation}
Then $\eta_i=1+e^{-\xi_i}/2-\sqrt{e^{-\xi_i}+e^{-2\xi_i}/4}$.
\end{enumerate}
\end{example}

\begin{example}
\label{IIIex:3a}
Let $f\colon(\xi_i)_{1\leq i\leq m}\mapsto\sum_{i=1}^m
\vartheta(\xi_i)$, 
where $\vartheta$ is Hellinger-like function, i.e.,
\begin{equation}
\vartheta\colon\xi\mapsto
\begin{cases}
-\sqrt{1-\xi^2},&\text{if}\;\;
\xi\in\left[-1,1\right];\\
\pinf,&\text{otherwise},
\end{cases}
\end{equation}
let $\gamma\in\RPP$, and let $\varphi=f$. 
Since $f$ is a cofinite Legendre function 
\cite[Sections~5 and 6]{IIIBB97}, 
Proposition~\ref{IIIl:fprox}\ref{IIIl:fproxiib} 
asserts that $\dom\Prox_{\gamma\varphi}^f=\RR^m$. 
Let $(\xi_i)_{1\leq i\leq m}\in\RR^m$, and set 
$(\eta_i)_{1\leq i\leq m}
=\Prox_{\gamma\varphi}^f(\xi_i)_{1\leq i\leq m}$. 
Then $(\forall i\in\{1,\ldots,m\})$ 
$\eta_i=\xi_i/\sqrt{(\gamma+1)^2+\xi_i^2}$.
\end{example}

\begin{example}
\label{IIIex:4a}
Let $\gamma\in\RPP$, let $\phi\in\Gamma_0(\RR)$ be such that 
$\dom\phi\cap\RPP\neq\emp$, and let $\vartheta$ 
be Burg entropy, i.e.,
\begin{equation}
\vartheta\colon\xi\mapsto
\begin{cases}
-\ln\xi,&\text{if}\;\xi\in\RPP;\\
\pinf,&\text{otherwise}.
\end{cases}
\end{equation}
Set $\varphi\colon(\xi_i)_{1\leq i\leq m}\mapsto
\sum_{i=1}^m\phi(\xi_i)$ and $f\colon(\xi_i)_{1\leq i\leq m}
\mapsto\sum_{i=1}^m\vartheta(\xi_i)$, 
let $(\xi_i)_{1\leq i\leq m}\in\RR^m$, 
and set $(\eta_i)_{1\leq i\leq m}
=\Prox_{\varphi}^f(\xi_i)_{1\leq i\leq m}$. 
Let $i\in\{1,\ldots,m\}$. Then $\eta_i$ can be computed 
as follows.
\begin{enumerate}
\item
\label{IIex:4ai}
Suppose that $\phi=\vartheta$ and $\xi_i\in\left]-\infty,0\right]$. 
Then $\eta_i=-(1+\gamma)^{-1}\xi_i$.

\item
\label{IIex:4aii}
Suppose that $\phi\colon\xi\mapsto\alpha|\xi|$ and 
$\xi_i\in\left]-\infty,\gamma\alpha\right]$. 
Then $\eta_i=(\gamma\alpha-\xi_i)^{-1}$.
\end{enumerate}
\end{example}

The following result will be used subsequently.

\begin{lemma}
\label{IIIle:tech}
Let $\XX$ be a reflexive real Banach space, 
let $f\in\Gamma_0(\XX)$ be a Legendre function, 
let $x\in\IDD$, 
and let $(x_n)_{n\in\NN}\in(\IDD)^{\NN}$. 
Suppose that $(D^f(x,x_n))_{n\in\NN}$ is bounded, 
that $\dom f^*$ is open, and that $\nabla f^*$ 
is weakly sequentially continuous. 
Then $\mathfrak{W}(x_n)_{n\in\NN}\subset\IDD$.
\end{lemma}
\begin{proof}
\cite[Proof of Theorem~4.1]{IIIQ}.
\end{proof}

\section{Forward-backward splitting in Banach spaces}
\label{IIIsec:FB}
\noindent
The first result in this section is a version of 
the forward-backward 
splitting algorithm in reflexive real Banach spaces 
which employs different Bregman distance-based 
proximity operators over the iterations.

\begin{theorem}
\label{IIIt:1}
Consider the setting of Problem~\ref{IIIpb:1} 
and let $f\in\Gamma_0(\XX)$ be a Legendre function 
such that $\BFS\cap\IDD\neq\emp$, 
$L(\IDD)\subset\IDDPS$, 
and $f\succcurlyeq\beta\psi\circ L$ for some $\beta\in\RPP$. 
Let $(\eta_n)_{n\in\NN}\in\ell_{+}^1(\NN)$, let $\alpha\in\RPP$, 
and let $(f_n)_{n\in\NN}$ be Legendre functions in 
$\BP_{\alpha}(f)$ such that
\begin{equation}
(\forall n\in\NN)\quad(1+\eta_n)f_n\succcurlyeq f_{n+1}.
\end{equation} 
Suppose that either $-L^*(\ran\nabla\psi)\subset\dom\varphi^*$ 
or $(\forall n\in\NN)$ $f_n$ is cofinite. 
Let $\varepsilon\in\left]0,\alpha\beta/(\alpha\beta+1)\right[$ 
and let $(\gamma_n)_{n\in\NN}$ be a sequence in $\mathbb{R}$ 
such that 
\begin{equation}
\label{IIIe:3}
(\forall n\in\NN)\quad
\varepsilon\leq\gamma_n\leq\alpha\beta(1-\varepsilon)
\quad\text{and}\quad(1+\eta_n)\gamma_n-\gamma_{n+1}
\leq\alpha\beta\eta_n.
\end{equation}
Furthermore, let $x_0\in\IDD$ and iterate
\begin{equation}
\label{IIIal:1}
(\forall n\in\NN)\quad x_{n+1}=\Prox_{\gamma_n\varphi}^{f_n}
\big(\nabla f_n(x_n)-\gamma_nL^*\nabla\psi(Lx_n)\big).
\end{equation}
Suppose in addition that $(\forall x\in\IDD)$ 
$D^f(x,\cdot)$ is coercive. Then $(x_n)_{n\in\mathbb{N}}$ 
is a bounded sequence in $\IDD$ and 
$\mathfrak{W}(x_n)_{n\in\NN}\subset\BFS$. 
Moreover, there exists $\overline{x}\in\BFS$ 
such that the following hold.
\begin{enumerate}
\item
\label{IIIt:1ii}
Suppose that $\BFS\cap\overline{\dom}f$ is a singleton. 
Then $x_n\weakly\overline{x}$.
\item
\label{IIIt:1iii}
Suppose that there exists $g\in\BF(f)$ such that 
for every $n\in\NN$, $g\succcurlyeq f_n$, 
and that, for every $y_1\in\XX$ and every $y_2\in\XX$,
\begin{equation}
\label{IIIe:4}
\begin{cases}
y_1\in\mathfrak{W}(x_n)_{n\in\NN}\\
y_2\in\mathfrak{W}(x_n)_{n\in\NN}\\
\big(\Pair{y_1-y_2}{\nabla f_n(x_n)
-\gamma_nL^*\nabla\psi(Lx_n)}\big)_{n\in\NN}
\quad\text{converges}
\end{cases}
\Rightarrow\quad y_1=y_2.
\end{equation}
In addition, suppose that one of the following holds.
\begin{enumerate}
\item
\label{IIIt:1iiia}
$\BFS\subset\IDD$.
\item
\label{IIIt:1iiib}
$\dom f^*$ is open and $\nabla f^*$ is weakly 
sequentially continuous.
\end{enumerate}
Then $x_n\weakly\overline{x}$. 
\item
\label{IIIt:1iv} 
Suppose that $f$ satisfies Condition~\ref{IIIcd:2} 
and that one of the following holds.
\begin{enumerate}
\item
\label{IIIt:1iva}
$\varphi$ is uniformly convex at 
$\overline{x}$.
\item
\label{IIIt:1ivb}
$\psi$ is uniformly convex at 
$L\overline{x}$ and there exists $\kappa\in\RPP$ such that 
$(\forall x\in\XX)$ $\|Lx\|\geq\kappa\|x\|$.
\item
\label{IIIt:1ivc}
$\varliminf D^f_{\BFS}(x_n)=0$ and there exists $\mu\in\RPP$ 
such that $(\forall n\in\NN)$ $\mu f\succcurlyeq f_n$.
\end{enumerate}
Then $x_n\to\overline{x}$.
\end{enumerate}
\end{theorem}
\begin{proof}
We first derive from 
Proposition~\ref{IIIl:fprox}\ref{IIIl:fproxiic} 
that the operators $(\Prox_{\gamma_n\varphi}^f)_{n\in\NN}$ are 
single-valued on their domains. We also note that 
$x_0\in\IDD$. Suppose that $x_n\in\IDD$ for some 
$n\in\NN$. If $f_n$ is cofinite then 
Proposition~\ref{IIIl:fprox}\ref{IIIl:fproxiib} yields
\begin{equation}
\label{IIIe:4c}
\nabla f_n(x_n)-\gamma_n L^*\nabla\psi(Lx_n)
\in\XX^*=\dom\Prox_{\gamma_n\varphi}^{f_n}.
\end{equation} 
Otherwise,
\begin{align}
\label{IIIe:5}
\nabla f_n(x_n)-\gamma_n L^*\nabla\psi(Lx_n)
&\in\IDD_n^*+\gamma_n\dom\varphi^*
=\inte(\IDD_n^*+\gamma_n\dom\varphi^*)\nonumber\\
&\subset\inte(\dom f_n^*+\gamma_n\dom\varphi^*)
=\inte(\dom f_n^*+\dom(\gamma_n\varphi^*)).
\end{align}
Since $\inte(\dom f_n^*+\dom(\gamma_n\varphi^*))
\subset\dom\Prox_{\gamma_n\varphi}^f$ by 
Proposition~\ref{IIIl:fprox}\ref{IIIl:fproxiib}, 
we deduce from \eqref{IIIal:1}, \eqref{IIIe:4c}, \eqref{IIIe:5}, 
and Proposition~\ref{IIIl:fprox}\ref{IIIl:fproxiia} 
that $x_{n+1}$ is a well-defined 
element in $\ran\Prox_{\gamma\varphi}^{f_n}
=\dom\partial\varphi\cap\IDD_n=\dom\partial\varphi\cap\IDD\subset\IDD$. 
By reasoning by induction, we conclude that
\begin{equation}
\label{IIIe:4g}
(x_n)_{n\in\NN}\in(\IDD)^{\NN}
\quad\text{is well-defined}.
\end{equation}
Next, let us set $\Phi=\varphi+\psi\circ L$ and
\begin{equation}
\label{IIIe:5b}
\begin{aligned}
(\forall n\in\NN)\quad g_n\colon\XX&\to\RX\\ 
x&\mapsto
\begin{cases}
f_n(x)-\gamma_n\psi(Lx),&\text{if}\; x\in\IDD;\\
+\infty,&\text{otherwise}.
\end{cases}
\end{aligned}
\end{equation}
Since $L(\IDD)\subset\IDDPS$, it follows 
from \eqref{IIIe:5b} that $(\forall n\in\NN)$ 
$g_n$ is G\^ateaux differentiable on 
$\dom g_n=\IDDG_n=\IDD$. Since $\psi$ is 
continuous on $\IDDPS\supset L(\IDD)$ 
and the functions $(f_n)_{n\in\NN}$ are continuous on $\IDD$ 
\cite[Proposition~3.3]{IIIPhelps}, we deduce that 
$(\forall n\in\NN)$ $g_n$ is continuous on $\dom g_n$. 
In addition,
\begin{equation}
\label{IIIe:5c}
(\forall n\in\NN)\quad g_n-\varepsilon\alpha f=(1-\varepsilon)
(f_n-\alpha\beta\psi\circ L)+\varepsilon(f_n-\alpha f)
+\big(\alpha\beta(1-\varepsilon)-\gamma_n\big)\psi\circ L.
\end{equation}
Note that $f\succcurlyeq\beta\psi\circ L$ and $(\forall n\in\NN)$ 
$f_n\succcurlyeq\alpha f$. Hence, \eqref{IIIe:5c} yields
\begin{equation}
\label{IIIe:7}
(\forall n\in\NN)\quad f_n\succcurlyeq\alpha\beta\psi\circ L,
\end{equation}
and hence, we deduce from \eqref{IIIe:3} and \eqref{IIIe:5c} that 
$(\forall n\in\NN)$ $g_n\succcurlyeq\varepsilon\alpha f$. In turn,
\begin{multline}
(\forall n\in\NN)(\forall x\in\dom g_n)(\forall y\in\dom g_n)\quad
\Pair{x-y}{\nabla g_n(x)-\nabla g_n(y)}\\
=D^{g_n}(x,y)+D^{g_n}(y,x)\geq\varepsilon\alpha\big(D^f(x,y)+D^f(y,x)\big)
\geq 0,
\end{multline}
and it therefore follows from \cite[Theorem~2.1.11]{IIIZa02_B} that 
$(\forall n\in\NN)$ $g_n$ is convex. Consequently, 
\begin{equation}
\label{IIIe:6}
(\forall n\in\NN)\quad g_n\in\BP_{\varepsilon\alpha}(f).
\end{equation}
Set $\omega=1+1/\varepsilon$. Then
\begin{align}
\label{IIIe:8}
(\forall n\in\NN)\quad (1+\omega\eta_n)g_n-g_{n+1}
&=(1+\omega\eta_n)(f_n-\gamma_n\psi\circ L)-(f_{n+1}
-\gamma_{n+1}\psi\circ L)\nonumber\\
&=(1+\eta_n)f_n-f_{n+1}+\eta_n\varepsilon^{-1}
\left(f_n-(\gamma_n+\varepsilon\alpha\beta)\psi\circ L\right)\nonumber\\
&\;+\big(\alpha\beta\eta_n+\gamma_{n+1}-(1+\eta_n)\gamma_n\big)
\psi\circ L.
\end{align}
We thus derive from \eqref{IIIe:3} and \eqref{IIIe:7} that
\begin{equation}
\label{IIIe:9}
(\forall n\in\NN)\quad (1+\omega\eta_n)g_n\succcurlyeq g_{n+1}.
\end{equation}
By invoking \eqref{IIIal:1} and 
Proposition~\ref{IIIl:fprox}\ref{IIIl:fproxiia}, we get
\begin{equation}
(\forall n\in\NN)\quad\nabla f_n(x_n)-\gamma_nL^*\nabla\psi(Lx_n)
\in\nabla f_n(x_{n+1})+\gamma_n\partial\varphi(x_{n+1}),
\end{equation}
and therefore,
\begin{multline}
\label{IIIe:10}
(\forall n\in\NN)\quad 
\nabla f_n(x_n)-\gamma_nL^*\nabla\psi(Lx_n)
\in\nabla f_n(x_{n+1})-\gamma_n L^*\nabla\psi(Lx_{n+1})\\
+\gamma_n\big(\partial\varphi(x_{n+1})
+L^*\nabla\psi(Lx_{n+1})\big).
\end{multline}
Since \cite[Theorem~2.4.2(vii)--(viii)]{IIIZa02_B} yield
\begin{align}
\label{IIIe:11}
(\forall n\in\NN)\quad 
\partial\varphi(x_{n+1})+L^*\nabla\psi(Lx_{n+1})
&\subset\partial\varphi(x_{n+1})+L^*\big(\partial\psi(Lx_{n+1})
\big)\nonumber\\
&\subset\partial(\varphi+\psi\circ L)(x_{n+1})
=\partial\Phi(x_{n+1}),
\end{align}
we deduce from \eqref{IIIe:10} that
\begin{equation}
\label{IIIe:12}
(\forall n\in\NN)\quad \nabla g_n(x_n)-\nabla g_n(x_{n+1})
\in\gamma_n\partial\Phi(x_{n+1}).
\end{equation}
By appealing to \eqref{IIIe:subdiff} and \eqref{IIIe:12}, we get
\begin{equation}
\label{IIIe:mon}
(\forall x\in\dom\Phi\cap\dom f)(\forall n\in\NN)\quad 
\gamma_n^{-1}\Pair{x-x_{n+1}}{\nabla g_n(x_n)-\nabla g_n(x_{n+1})}
+\Phi(x_{n+1})
\leq\Phi(x),
\end{equation}
and hence, by \cite[Proposition~2.3(ii]{IIIBBC03},
\begin{multline}
\label{IIIe:14t}
(\forall x\in\dom\Phi\cap\dom f)(\forall n\in\NN)\quad 
\gamma_n^{-1}\big(D^{g_n}(x,x_{n+1})+D^{g_n}(x_{n+1},x_n)
-D^{g_n}(x,x_n)\big)\\
+\Phi(x_{n+1})\leq\Phi(x).
\end{multline}
In particular,
\begin{equation}
\label{IIIe:14}
(\forall x\in\BFS\cap\dom f)(\forall n\in\NN)\quad 
D^{g_n}(x,x_{n+1})+D^{g_n}(x_{n+1},x_n)
-D^{g_n}(x,x_n)\leq 0.
\end{equation}
By using \eqref{IIIe:9}, we deduce from \eqref{IIIe:14} that
\begin{multline}
\label{IIIe:15}
(\forall x\in\BFS\cap\dom f)(\forall n\in\NN)\quad 
D^{g_{n+1}}(x,x_{n+1})+(1+\omega\eta_n)D^{g_n}(x_{n+1},x_n)\\
\leq (1+\omega\eta_n)D^{g_n}(x,x_n),
\end{multline}
and therefore,
\begin{equation}
\label{IIIe:15t}
(\forall x\in\BFS\cap\dom f)(\forall n\in\NN)\quad 
D^{g_{n+1}}(x,x_{n+1})\leq (1+\omega\eta_n)D^{g_n}(x,x_n).
\end{equation}
This shows that $(x_n)_{n\in\NN}$ is stationarily 
quasi-Bregman monotone with respect to $\BFS$ relative 
to $(g_n)_{n\in\NN}$. Hence, we deduce from 
Proposition~\ref{IIIpp:qb1}\ref{IIIpp:qb1ii} 
that
\begin{equation}
\label{IIIe:bound}
(x_n)_{n\in\NN}\in(\IDD)^{\NN}\quad\text{is bounded}
\end{equation}
and, since $\XX$ is reflexive,
\begin{equation}
\label{IIIe:ex}
\mathfrak{W}(x_n)_{n\in\NN}\neq\emp.
\end{equation}
In addition, we derive from \eqref{IIIe:15t} and 
Proposition~\ref{IIIpp:qb1}\ref{IIIpp:qb1i} that
\begin{equation}
\label{IIIe:16}
(\forall x\in\BFS\cap\IDD)\quad\big(D^{g_n}(x,x_n)\big)_{n\in\NN}
\quad\text{converges,}
\end{equation}
and thus, since \eqref{IIIe:15} yields
\begin{align}
(\forall x\in\BFS\cap\IDD)(\forall n\in\NN)\quad 
0&\leq D^{g_n}(x_{n+1},x_n)\nonumber\\
&\leq (1+\omega\eta_n)D^{f_n}(x_{n+1},x_n)\nonumber\\
&\leq (1+\omega\eta_n)D^{f_n}(x,x_n)-D^{f_{n+1}}(x,x_{n+1}),
\end{align}
we obtain
\begin{equation}
\label{IIIe:17}
D^{g_n}(x_{n+1},x_n)\to 0.
\end{equation}
On the other hand, it follows from \eqref{IIIe:6} that
\begin{equation}
(\forall n\in\NN)\quad\varepsilon\alpha D^f(x_{n+1},x_n)
\leq D^{g_n}(x_{n+1},x_n),
\end{equation}
and hence, \eqref{IIIe:17} yields
\begin{equation}
\label{IIIe:18}
D^f(x_{n+1},x_n)\to 0.
\end{equation}
Now, it follows from \eqref{IIIe:14t} that
\begin{equation}
(\forall n\in\NN)\quad
\Phi(x_{n+1})\leq\gamma_n^{-1}\big(D^{g_n}(x_n,x_{n+1})
+D^{g_n}(x_{n+1},x_n)\big)+\Phi(x_{n+1})\leq\Phi(x_n),
\end{equation}
which shows that $(\Phi(x_n))_{n\in\NN}$ is decreasing 
and hence, since it is bounded from below by $\inf\Phi(\XX)$, 
it is convergent. However, \eqref{IIIe:14t} 
and \eqref{IIIe:15t} yield
\begin{align}
\label{IIIe:20}
&(\forall x\in\BFS\cap\IDD)(\forall n\in\NN)\nonumber\\
&\varepsilon^{-1}
\Bigg(\dfrac{1}{1+\omega\eta_n}D^{g_{n+1}}(x,x_{n+1})
+D^{g_n}(x_{n+1},x_n)-D^{g_n}(x,x_n)\Bigg)
+\Phi(x_{n+1})\nonumber\\
&\leq\gamma_n^{-1}
\Bigg(\dfrac{1}{1+\omega\eta_n}D^{g_{n+1}}(x,x_{n+1})
+D^{g_n}(x_{n+1},x_n)-D^{g_n}(x,x_n)\Bigg)
+\Phi(x_{n+1})\nonumber\\
&\leq\Phi(x).
\end{align}
Since $\eta_n\to 0$, by taking the limit in 
\eqref{IIIe:20} and then 
using \eqref{IIIe:16} and \eqref{IIIe:17}, we get
\begin{equation}
\inf\Phi(\XX)\leq\lim\Phi(x_n)\leq\inf\Phi(\XX),
\end{equation}
and thus, 
\begin{equation}
\label{IIIe:20b1}
\Phi(x_n)\to\inf\Phi(\XX).
\end{equation}
We now show that
\begin{equation}
\label{IIIe:sol}
\mathfrak{W}(x_n)_{n\in\NN}\subset\BFS.
\end{equation}
To this end, suppose that $x\in\mathfrak{W}(x_n)_{n\in\NN}$, 
i.e., $x_{k_n}\weakly x$. Since $\Phi$ is weakly 
lower semicontinuous \cite[Theorem~2.2.1]{IIIZa02_B}, 
by \eqref{IIIe:20b1},
\begin{equation}
\inf\Phi(\XX)\leq\Phi(x)\leq\varliminf
\Phi(x_{k_n})=\lim\Phi(x_n)=\inf\Phi(\XX).
\end{equation}
This yields $\Phi(x)=\inf\Phi(\XX)$, 
i.e., $x\in\Argmin\Phi=\BFS$.

\ref{IIIt:1ii}:
Let $\overline{x}\in\mathfrak{W}(x_n)_{n\in\NN}$. 
Since \eqref{IIIe:bound} and \eqref{IIIe:sol} 
imply that $\mathfrak{W}(x_n)_{n\in\NN}
\subset\BFS\cap\overline{\dom}f$, we obtain 
$\mathfrak{W}(x_n)_{n\in\NN}=\{\overline{x}\}$, 
and in turn, \eqref{IIIe:ex} yields $x_n\weakly\overline{x}$.

\ref{IIIt:1iii}: 
In view of \eqref{IIIe:sol} and Proposition~\ref{IIIpp:qb2}, 
it suffices to show that $\mathfrak{W}(x_n)_{n\in\NN}\subset\IDD$.

\ref{IIIt:1iiia}: 
We have $\mathfrak{W}(x_n)_{n\in\NN}\subset\BFS\subset\IDD$.

\ref{IIIt:1iiib}: 
This follows from Lemma~\ref{IIIle:tech}.

\ref{IIIt:1iv}: 
Let $\overline{x}\in\BFS\cap\IDD$. 
Since $f$ satisfies Condition~\ref{IIIcd:2}, 
\eqref{IIIe:18} yields
\begin{equation}
\label{IIIe:28}
x_{n+1}-x_n\to 0.
\end{equation}
Now set
\begin{equation}
(\forall n\in\NN)\quad z_n=x_{n+1}
\quad\text{and}\quad 
z_n^*=\gamma_n^{-1}\big(\nabla g_n(x_n)-\nabla g_n(z_n)\big).
\end{equation} 
Then \eqref{IIIe:12} and \eqref{IIIe:28} imply that 
\begin{equation}
\label{IIIe:29}
(\forall n\in\NN)\quad z_n^*\in\partial\Phi(z_n)
\quad\text{and}\quad 
z_n-x_n\to 0.
\end{equation} 
Since \eqref{IIIe:15} yields
\begin{align}
(\forall n\in\NN)\quad D^{g_{n+1}}(\overline{x},x_{n+1})
&=D^{g_{n+1}}(\overline{x},z_n)\nonumber\\
&\leq(1+\omega\eta_n)D^{g_n}(\overline{x},z_n)\nonumber\\
&=(1+\omega\eta_n)D^{g_n}(\overline{x},x_{n+1})\nonumber\\
&\leq(1+\omega\eta_n)D^{g_n}(\overline{x},x_n),
\end{align}
we deduce that
\begin{equation}
\label{IIIe:30g}
(\forall n\in\NN)\quad(1+\omega\eta_n)^{-1}
D^{g_{n+1}}(\overline{x},x_{n+1})\leq D^{g_n}(\overline{x},z_n)
\leq D^{g_n}(\overline{x},x_n).
\end{equation}
Altogether, \eqref{IIIe:16} and \eqref{IIIe:30g} yield
\begin{equation}
\label{IIIe:31}
D^{g_n}(\overline{x},z_n)-D^{g_n}(\overline{x},x_n)\to 0.
\end{equation}
In \eqref{IIIe:mon}, by setting $x=\overline{x}$, we get
\begin{align}
\label{IIIe:32}
(\forall n\in\NN)\quad 
0&\leq\gamma_n\pair{z_n-\overline{x}}{z_n^*}\nonumber\\
&=\Pair{z_n-\overline{x}}{\nabla g_n(x_n)-\nabla g_n(z_n)}
\nonumber\\
&=D^{g_n}(\overline{x},x_n)
-D^{g_n}(\overline{x},z_n)-D^{g_n}(z_n,x_n)\nonumber\\
&\leq D^{g_n}(\overline{x},x_n)-D^{g_n}(\overline{x},z_n).
\end{align}
By taking to the limit in \eqref{IIIe:32} and using 
\eqref{IIIe:31}, we get 
\begin{equation}
\label{IIIe:33}
\pair{z_n-\overline{x}}{z_n^*}\to 0.
\end{equation}

\ref{IIIt:1iva}: 
In this case $\BFS=\{\overline{x}\}$. 
Since $\varphi$ is uniformly convex at $\overline{x}$, 
$\Phi$ is likewise and hence, there exists an increasing function 
$\phi\colon\left[0,+\infty\right[\to
\left[0,+\infty\right]$ that vanishes only at $0$ such that
\begin{equation}
\label{IIIe:34}
(\forall n\in\NN)(\forall\tau\in\left]0,1\right[)\quad 
\Phi(\tau\overline{x}+(1-\tau)z_n)+\tau(1-\tau)
\phi(\|z_n-\overline{x}\|)
\leq\tau\Phi(\overline{x})+(1-\tau)\Phi(z_n).
\end{equation}
It therefore follows from \cite[Page~201]{IIIZa02_B} that 
$\partial\Phi$ is uniformly monotone at 
$\overline{x}$ and its modulus of convexity is $\phi$, i.e,
\begin{equation}
\label{IIIe:34c}
(\forall n\in\NN)\quad\pair{z_n-\overline{x}}{z_n^*}
\geq\phi(\|z_n-\overline{x}\|)\geq 0.
\end{equation}
Altogether, \eqref{IIIe:33} and \eqref{IIIe:34c} 
yield $\phi(\|z_n-\overline{x}\|)\to 0$, 
and thus, $z_n\to\overline{x}$. In turn, \eqref{IIIe:29} 
yields $x_n\to\overline{x}$.

\ref{IIIt:1ivb}: 
By the same argument as in \ref{IIIt:1iva}, 
$\BFS=\{\overline{x}\}$ 
and there exists an increasing function 
$\phi\colon\left[0,+\infty\right[\to
\left[0,+\infty\right]$ that vanishes only at $0$ such that
\begin{equation}
(\forall n\in\NN)\quad
\Pair{z_n-\overline{x}}{\nabla\psi(Lz_n)-\nabla\psi(L\overline{x}}\geq 
\phi(\|Lz_n-L\overline{x}\|).
\end{equation}
In turn, it follows from \eqref{IIIe:11} and 
\cite[Theorem~2.4.2(iv)]{IIIZa02_B} that
\begin{equation}
(\forall n\in\NN)\quad\pair{z_n-\overline{x}}{z_n^*}\geq 
\phi(\|Lz_n-L\overline{x}\|).
\end{equation}
This yields $\phi(\|Lz_n-L\overline{x}\|)
\to 0$, and hence, $Lz_n\to L\overline{x}$. 
Since
\begin{equation}
(\forall n\in\NN)\quad\|Lz_n-L\overline{x}\|
\geq\kappa\|z_n-\overline{x}\|,
\end{equation}
we obtain $z_n\to\overline{x}$ and in turn, 
\eqref{IIIe:29} yields $x_n\to\overline{x}$.

\ref{IIIt:1ivc}: 
First, we observe that $\BFS$ is closed and convex 
since $\Phi\in\Gamma_0(\XX)$. 
Next, for every $n\in\NN$, since 
$\mu f\succcurlyeq f_n$, we derive from \eqref{IIIe:5b} 
that $\mu f\succcurlyeq g_n$. 
Finally, the strong convergence follows from 
Proposition~\ref{IIIpp:qb3}.
\end{proof}

The following corollary of Theorem~\ref{IIIt:1} appears to be 
the first version of the forward-backward algorithm 
outside of Hilbert spaces.

\begin{theorem}
\label{IIIt:2}
Consider the setting of Problem~\ref{IIIpb:1} and let 
$f\in\Gamma_0(\XX)$ be a Legendre function such that 
$\BFS\cap\IDD\neq\emp$, 
$L(\IDD)\subset\IDDPS$, 
and $f\succcurlyeq\beta\psi\circ L$ for some $\beta\in\RPP$. 
Suppose that either $f$ is cofinite or 
$-L^*(\ran\nabla\psi)\subset\dom\varphi^*$, 
and that $(\forall x\in\IDD)$ $D^f(x,\cdot)$ 
is coercive. Let $\varepsilon\in\left]0,\beta/(\beta+1)\right[$, 
let $(\eta_n)_{n\in\NN}\in\ell_{+}^1(\NN)$, 
and let $(\gamma_n)_{n\in\NN}$ be a sequence in $\mathbb{R}$ 
such that 
\begin{equation}
\label{IIIe:35}
(\forall n\in\NN)\quad
\varepsilon\leq\gamma_n\leq\beta(1-\varepsilon)
\quad\text{and}\quad
(1+\eta_n)\gamma_n-\gamma_{n+1}\leq\beta\eta_n.
\end{equation}
Furthermore, let $x_0\in\IDD$ and iterate
\begin{equation}
\label{IIIal:2}
(\forall n\in\NN)\quad x_{n+1}=\Prox_{\gamma_n\varphi}^f
\big(\nabla f(x_n)-\gamma_nL^*\nabla\psi(Lx_n)\big).
\end{equation}
Then there exists $\overline{x}\in\BFS$ 
such that the following hold.
\begin{enumerate}
\item
\label{IIIt:2i}
Suppose that one of the following holds.
\begin{enumerate}
\item
\label{IIIt:2ia}
$\BFS\cap\overline{\dom}f$ is a singleton.
\item
\label{IIIt:2ib}
$\nabla f$ and $\nabla\psi$ are weakly sequentially 
continuous and $\BFS\subset\IDD$.
\item
\label{IIIt:2ic}
$\dom f^*$ is open and $\nabla f$, $\nabla f^*$, 
and $\nabla\psi$ are weakly sequentially continuous.
\end{enumerate}
Then $x_n\weakly\overline{x}$.
\item
\label{IIIt:2ii}
Suppose that $f$ satisfies Condition~\ref{IIIcd:2} 
and that one of the following holds.
\begin{enumerate}
\item
\label{IIIt:2iia}
$\varphi$ is uniformly convex at $\overline{x}$.
\item
\label{IIIt:2iib}
$\psi$ is uniformly convex at $L\overline{x}$ and there exists 
$\kappa\in\RPP$ such that $(\forall x\in\XX)$ 
$\|Lx\|\geq\kappa\|x\|$.
\item
\label{IIIt:2iic} 
$\varliminf D^f_{\BFS}(x_n)=0$.
\end{enumerate}
Then $x_n\to\overline{x}$.
\end{enumerate}
\end{theorem}
\begin{proof}
Set $(\forall n\in\NN)$ $f_n=f$. Then 
\begin{equation}
\label{IIIe:35b}
(\forall n\in\NN)\quad 
\begin{cases}
f_n\in\BP_1(f),\\
f\succcurlyeq f_n,\\
(1+\eta_n)f_n\succcurlyeq f_{n+1}.
\end{cases}
\end{equation}

\ref{IIIt:2ia}: 
This is a corollary of Theorem~\ref{IIIt:1}\ref{IIIt:1ii}.

\ref{IIIt:2ib}--\ref{IIIt:2ic}: 
Firstly, the proof of 
Theorem~\ref{IIIt:1}\ref{IIIt:1iiia}--\ref{IIIt:1iiib} 
shows that $\mathfrak{W}(x_n)_{n\in\NN}\subset\IDD$. 
Next, in view of Theorem~\ref{IIIt:1}\ref{IIIt:1iii}, 
it suffices to show that \eqref{IIIe:4} holds. To this end, 
suppose that $y_1$ and $y_2$ are two weak sequential cluster 
points of $(x_n)_{n\in\NN}$ such that
\begin{equation}
\label{IIIe:36}
\big(\Pair{y_1-y_2}{\nabla f(x_n)
-\gamma_nL^*\nabla\psi(Lx_n)}\big)_{n\in\NN}
\quad\text{converges}.
\end{equation}
Then, there exist two strictly increasing sequences 
$(k_n)_{n\in\NN}$ and $(l_n)_{n\in\NN}$ in $\NN$ such that 
$x_{k_n}\weakly y_1$ and $x_{l_n}\weakly y_2$. 
We derive from \eqref{IIIe:35} and 
\cite[Lemma~2.2.2]{IIIPo87_B} that there 
exists $\theta\in\left[\varepsilon,\beta(1-\varepsilon)
\right]$ such that $\gamma_n\to\theta$. Since $\nabla f$ and 
$\nabla\psi$ are weakly sequentially continuous, after taking 
the limit in \eqref{IIIe:36} along the subsequences 
$(x_{k_n})_{n\in\NN}$ and $(x_{l_n})_{n\in\NN}$, respectively, 
we get 
\begin{equation}
\label{IIIe:37}
\Pair{y_1-y_2}{\nabla f(y_1)-\theta L^*\nabla\psi(Ly_1)}
=\Pair{y_1-y_2}{\nabla f(y_2)-\theta L^*\nabla\psi(Ly_2)}.
\end{equation}
Let us define
\begin{equation}
\begin{aligned}
h\colon\XX&\to\RX\\
x&\mapsto
\begin{cases}
f(x)-\theta\psi(Lx),&\text{if}\; x\in\IDD;\\
+\infty,&\text{otherwise}.
\end{cases}
\end{aligned}
\end{equation}
Then $h$ is G\^ateaux differentiable on $\inte\dom h=\IDD$ 
and \eqref{IIIe:37} yields
\begin{equation}
\label{IIIe:38}
\Pair{y_1-y_2}{\nabla h(y_1)-\nabla h(y_2)}=0.
\end{equation}
On the other hand,
\begin{equation}
h-\varepsilon f=f-\theta\psi\circ L-\varepsilon f
=(1-\varepsilon)(f-\beta\psi\circ L)
+\big(\beta(1-\varepsilon)-\theta\big)\psi\circ L.
\end{equation}
In turn, since $f\succcurlyeq\beta\psi\circ L$ and 
$\theta\leq\beta(1-\varepsilon)$, we obtain 
$h\succcurlyeq\varepsilon f$, and hence,
\begin{equation}
D^h(y_1,y_2)\geq\varepsilon D^f(y_1,y_2)\quad\text{and}\quad 
D^h(y_2,y_1)\geq\varepsilon D^f(y_2,y_1).
\end{equation}
Therefore, \eqref{IIIe:38} yields
\begin{align}
\label{IIIe:40t}
0&=\Pair{y_1-y_2}{\nabla h(y_1)-\nabla h(y_2)}\nonumber\\
&=D^h(y_1,y_2)+D^h(y_2,y_1)\nonumber\\
&\geq\varepsilon\big(D^f(y_1,y_2)+D^f(y_2,y_1)\big)\nonumber\\
&=\varepsilon\Pair{y_1-y_2}{\nabla f(y_1)-\nabla f(y_2)}.
\end{align}
Suppose that $y_1\neq y_2$. 
Since $f|_{\IDD}$ is strictly convex, $\nabla f$ is 
strictly monotone \cite[Theorem~2.4.4(ii)]{IIIZa02_B}, 
i.e.,
\begin{equation}
\label{IIIe:40}
\Pair{y_1-y_2}{\nabla f(y_1)-\nabla f(y_2)}>0
\end{equation}
and we reach a contradiction.

\ref{IIIt:2ii}: 
The conclusions follow from \eqref{IIIe:35b} 
and Theorem~\ref{IIIt:1}\ref{IIIt:1iv}.
\end{proof}

\begin{remark}
In Problem~\ref{IIIpb:1}, suppose that $L=\Id$. 
We rewrite algorithm~\eqref{IIIal:2} as follow
\begin{equation}
\label{IIIe:42}
(\forall n\in\NN)\quad x_{n+1}
=\argmind{x\in\XX}{\Big(\varphi(x)+
\Pair{x-x_n}{\nabla\psi(x_n)}+\psi(x_n)
+\gamma_n^{-1}D^f(x,x_n)\Big)}.
\end{equation}
Another method to solve Problem~\ref{IIIpb:1} 
was proposed in \cite{IIIBred09}. In 
that method, instead of solving \eqref{IIIe:42}, 
the authors solve
\begin{equation}
\label{IIIe:43}
(\forall n\in\NN)\quad x_{n+1}
=\argmind{x\in\XX}{\Big(\varphi(x)+
\Pair{x-x_n}{\nabla\psi(x_n)}+\psi(x_n)
+\gamma_n^{-1}\|x-x_n\|^p\Big)},
\end{equation}
for some $1<p\leq 2$. The weak convergence is established 
under the assumptions that Problem~\ref{IIIpb:1} admits 
a unique solution, $\nabla\psi$ is $(p-1)$-H\"{o}lder 
continuous with constant $\beta$, and 
$0<\inf_{n\in\NN}\gamma_n\leq\sup_{n\in\NN}\gamma_n
\leq(1-\delta)/\beta$, where $0<\delta<1$. 
The high nonlinearity of the regularization in \eqref{IIIe:43} 
compared to \eqref{IIIe:42} makes the numerical 
implementation of this method difficult in general. 
Furthermore, since \eqref{IIIe:43} yields
\begin{equation}
(\forall n\in\NN)\quad 0\in\partial\varphi(x_{n+1})+\nabla\psi(x_n)
+\gamma_n^{-1}\partial\big(\|x_{n+1}-x_n\|^p\big),
\end{equation}
and since $(\forall n\in\NN)$ $\partial\big(\|x_{n+1}-x_n\|^p\big)$ 
is not separable, this method is not a splitting method.
\end{remark}

\begin{remark}
We can reformulate Problem~\ref{IIIpb:1} as the following joint 
minimization problem
\begin{equation}
\minimize{(x,y)\in V}{\varphi(x)+\psi(y)},
\end{equation}
where $V=\gra L=\menge{(x,y)\in\XX\times\YY}{y=Lx}$. 
This constrained problem is equivalent to the following 
unconstrained problem
\begin{equation}
\minimize{(x,y)\in\XX\times\YY}{\varphi(x)+\psi(y)+\iota _V(x,y)}.
\end{equation}
In \cite{BCN06}, a different coupling term between the variables $x$
and $y$ was considered and the problem considered there was 
\begin{equation}
\minimize{(x,y)\in\XX\times\YY}{\varphi(x)+\psi(y)+D^f(x,y)},
\end{equation}
in the Euclidean spaces. Their method activates $\varphi$ and $\psi$
via their so-called left and right Bregman proximity operators 
alternatively (see also \cite{BC03} for the projection setting).
This method does not require the smoothness of $\psi$ but 
it requires the computation of Bregman distance-based 
proximity operator of $\psi$.
\end{remark}

Next, we provide a particular instance of Theorem~\ref{IIIt:2} 
in finite-dimensional spaces.

\begin{corollary}
\label{IIIc:1}
In the setting of Problem~\ref{IIIpb:1}, 
suppose that $\XX$ and $\YY$ are finite-dimensional. 
Let $f\in\Gamma_0(\XX)$ be a Legendre function such that 
$\BFS\cap\IDD\neq\emp$, $L(\IDD)\subset\IDDPS$, 
$f\succcurlyeq\beta\psi\circ L$ 
for some $\beta\in\RPP$, and $\dom f^*$ is open. 
Suppose that either $f$ is cofinite or 
$-L^*(\ran\nabla\psi)\subset\dom\varphi^*$. 
Let $\varepsilon\in\left]0,\beta/(\beta+1)\right[$, 
let $(\eta_n)_{n\in\NN}\in\ell_{+}^1(\NN)$, 
and let $(\gamma_n)_{n\in\NN}$ be a sequence in $\mathbb{R}$ 
such that 
\begin{equation}
\label{IIIe:35a}
(\forall n\in\NN)\quad
\varepsilon\leq\gamma_n\leq\beta(1-\varepsilon)
\quad\text{and}\quad
(1+\eta_n)\gamma_n-\gamma_{n+1}\leq\beta\eta_n.
\end{equation}
Furthermore, let $x_0\in\IDD$ and iterate
\begin{equation}
\label{IIIal:2a}
(\forall n\in\NN)\quad x_{n+1}=\Prox_{\gamma_n\varphi}^f
\big(\nabla f(x_n)-\gamma_nL^*\nabla\psi(Lx_n)\big).
\end{equation}
Then there exists $\overline{x}\in\BFS$ 
such that $x_n\to\overline{x}$.
\end{corollary}
\begin{proof}
Since $\dom f^*$ is open, \cite[Lemma~7.3(ix)]{IIIBBC01} 
asserts that $(\forall x\in\IDD)$ $D^f(x,\cdot)$ is coercive. 
Hence, the claim follows from Theorem~\ref{IIIt:2}\ref{IIIt:2ic}.
\end{proof}

\begin{remark}
We provide some special cases of Problem~\ref{IIIpb:1} 
and Theorem~\ref{IIIt:2}.
\begin{enumerate}
\item
Let $I$ and $K$ be totally ordered countable 
index sets. In Problem~\ref{IIIpb:1}, suppose that 
$\XX$ and $\YY$ are separable Hilbert spaces, and that 
$\psi\colon y\mapsto\sum_{k\in K}|\pair{y-r}{y_k}|^2/2$, 
where $r\in\YY$ and $(y_k)_{k\in K}$ 
is a frame in $\YY$, i.e.,
\begin{equation}
(\exists(\mu,\nu)\in\RPP^2)(\forall y\in\YY)\quad 
\mu\|y\|^2\leq\sum_{k\in K}|\pair{y}{y_k}|^2\leq\nu\|y\|^2.
\end{equation}
Then in Theorem~\ref{IIIt:2}, we can choose 
$f\colon x\mapsto\sum_{i\in I}|\pair{x}{x_i}|^2/2$, 
where $(x_i)_{i\in I}$ is a frame in $\XX$, i.e.,
\begin{equation}
(\exists(\alpha,\beta)\in\RPP^2)(\forall x\in\XX)\quad 
\alpha\|x\|^2\leq\sum_{i\in I}|\pair{x}{x_i}|^2\leq\beta\|x\|^2.
\end{equation}
It follows from \cite[Corollary~1]{IIICR14} 
that $f$ and $\psi$ are Legendre functions and that 
$\nabla f$ and $\nabla\psi$ are weakly sequentially continuous. 
Now let $x$ and $z$ be in $\XX$. Then 
\begin{align}
\label{IIIe:61}
D^{\psi}(Lx,Lz)
&=\sum_{k\in K}|\pair{Lx-Lz}{y_k}|^2/2
\leq\nu\|Lx-Lz\|^2/2\nonumber\\
&\leq\nu\|L\|^2\|x-z\|^2/2
\leq\nu|L\|^2\alpha^{-1}\sum_{i\in I}
|\pair{x-z}{x_i}|^2/2\nonumber\\
&=\nu\|L\|^2\alpha^{-1}D^f(x,z),
\end{align}
which implies that 
$f\succcurlyeq\alpha\nu^{-1}\|L\|^{-2}\psi\circ L$ 
and in addition, $D^f(x,\cdot)$ is coercive.
\item
Let $p$ and $q$ be in $\left]1,+\infty\right[$ 
and set $p^*=p/(p-1)$ and $q^*=q(p-1)$. 
In Problem~\ref{IIIpb:1}, suppose that 
$\XX=\ell^p(\NN)$ and $\YY=\ell^q(\NN)$, 
that $r\in\ell^q(\NN)$, 
that $\psi\colon y\mapsto\|y\|^q/q-\pair{y-r}{\|r\|^{q-2}r}
-\|r\|^q/q$, and that there exists 
$\kappa\in\RPP$ such that $(\forall x\in\ell^p(\NN))$ 
$\|Lx\|\geq\kappa\|x\|$. It follows from 
\cite[Theorem~4.7]{IIICi90_B} that 
$\ell^p(\NN)$ is uniformly convex and hence, 
strictly convex. Therefore, $\psi$ is 
strictly convex and supercoercive. 
The property of $L$ implies that 
$\psi\circ L$ is strictly convex and 
supercoercive, and $\varphi+\psi\circ L$ 
is likewise. In turn, Problem~\ref{IIIpb:1} 
admits a unique solution. 
Let $f\in\Gamma_0(\ell^p(\NN))$ be a cofinite Legendre 
function such that $f\succcurlyeq\beta\|L\cdot\|^p$ 
for some $\beta\in\RPP$, let $x\in\ell^p(\NN)$, 
and set
\begin{equation}
\phi\colon[0,+\infty[\to [0,+\infty]\colon 
t\mapsto\inf\limits_{\|z-x\|=t}
\big(\|x\|^p-p\pair{x-z}{\|z\|^{p-2}z}-\|z\|^p\big).
\end{equation}
Then
\begin{equation}
\label{IIIe:77}
(\forall z\in\ell^p(\NN))\quad
D^f(x,z)\geq p\beta D^{\psi}(Lx,Lz)\geq\beta\phi(\|Lz-Lx\|).
\end{equation}
If $p\in\left[2,+\infty\right[$ then we derive from 
\cite[Lemma~1.4.10]{IIIBI00_B} that
\begin{equation}
(\forall z\in\ell^p(\NN))
\quad\phi(\|Lz-Lx\|)\geq 2^{1-p}\|Lz-Lx\|^p\geq 
2^{1-p}\kappa^p\|z-x\|^p.
\end{equation}
Thus, it follows from \eqref{IIIe:77} that
\begin{equation}
(\forall z\in\ell^p(\NN))\quad 
D^f(x,z)\geq 2^{1-p}\kappa^p\beta\|z-x\|^p,
\end{equation}
and $D^f(x,\cdot)$ is therefore coercive. 
If $p\in\left]1,2\right[$ 
then we derive from \cite[Lemma~1.4.8]{IIIBI00_B} that 
\begin{equation}
\label{IIIe:78}
(\forall z\in\ell^p(\NN))\quad
\phi(\|Lz-Lx\|)\geq(\|Lz-Lx\|+\|Lz\|)^p-\|Lz\|^p
-p\|Lz\|^{p-1}\|Lz-Lx\|,
\end{equation}
and hence, \eqref{IIIe:77} yields
\begin{equation}
\label{IIIe:79}
(\forall z\in\ell^p(\NN))\; D^f(x,z)
\geq\beta\left((\|Lz\|+\|Lz-Lx\|)^p-\|Lz\|^p
-p\|Lz\|^{p-1}\|Lz-Lx\|\right).
\end{equation}
On the other hand, since
\begin{equation}
\label{IIIe:80}
\lim\limits_{\|Lz\|\to+\infty}\frac{\left(\|Lz\|
+\|Lz-Lx\|\right)^p-\|Lz\|^p-p\|Lz\|^{p-1}\|Lz-Lx\|}
{(2^p-1-p)\|Lz\|^p}=1,
\end{equation}
and since $2^p-1-p>0$, it follows from and \eqref{IIIe:79} 
and the property of $L$ that
\begin{equation}
\lim\limits_{\|z\|\to +\infty}D^f(x,z)=+\infty,
\end{equation}
and $D^f(x,\cdot)$ is therefore coercive. 
Consequently, Theorem~\ref{IIIt:2}\ref{IIIt:2ia} 
can be applied.
\end{enumerate}
\end{remark}

\section{Application to multivariate minimization}
\label{IIIsect:mul}
\noindent
We propose a variant of the forward-backward 
algorithm to solve the following multivariate 
minimization problem.

\begin{problem}
\label{IIIpb:2}
Let $m$ and $p$ be strictly positive integers, let 
$(\mathcal{X}_i)_{1\leq i\leq m}$ 
and $(\mathcal{Y}_k)_{1\leq k\leq p}$ 
be reflexive real Banach spaces. 
For every $i\in\{1,\ldots,m\}$ 
and every $k\in\{1,\ldots,p\}$, let 
$\varphi_i\in\Gamma_0(\XX_i)$, 
let $\psi_k\in\Gamma_0(\YY_k)$ be G\^ateaux 
differentiable on $\IDDPS_k\neq\emp$, 
and let $L_{ik}\colon\mathcal{X}_i\to\mathcal{Y}_k$ 
be linear and bounded. The problem is to
\begin{equation}
\label{IIIe:44}
\minimize{x_1\in\XX_1,\ldots,x_m\in\XX_m}
{\sum_{i=1}^m\varphi_i(x_i)+\sum_{k=1}^p\psi_k
\left(\sum_{i=1}^mL_{ik}x_i\right)}.
\end{equation}
Denote by $\BFS$ the set of solutions to \eqref{IIIe:44}.
\end{problem}

We derive from Theorem~\ref{IIIt:2} the following result.

\begin{proposition}
\label{IIIpp:2}
Consider the setting of Problem~\ref{IIIpb:2}. 
For every $k\in\{1,\ldots,p\}$, suppose that 
there exists $\sigma_k\in\RPP$ such that for every 
$(y_{ik})_{1\leq i\leq m}\in\IDDPS_k$ 
and every $(v_{ik})_{1\leq i\leq m}\in\IDDPS_k$ 
satisfying $\sum_{i=1}^my_{ik}\in\IDDPS_k$ 
and $\sum_{i=1}^mv_{ik}\in\IDDPS_k$, one has
\begin{equation}
\label{IIIe:45}
D^{\psi_k}\Big(\sum_{i=1}^my_{ik},\sum_{i=1}^mv_{ik}\Big)
\leq\sigma_k\sum_{i=1}^mD^{\psi_k}(y_{ik},v_{ik}).
\end{equation}
For every $i\in\{1,\ldots,m\}$, 
let $f_i\in\Gamma_0(\XX_i)$ be a Legendre function 
such that $(\forall x_i\in\IDD_i)$ $D^{f_i}(x_i,\cdot)$ 
is coercive. For every $k\in\{1,\ldots,p\}$, 
suppose that $\sum_{i=1}^mL_{ik}(\IDD_i)
\subset\IDDPS_k$, that, for every 
$i\in\{1,\ldots,m\}$, 
there exists $\beta_{ik}\in\RPP$ such that 
$f_i\succcurlyeq\beta_{ik}\psi_k\circ L_{ik}$, and 
set $\beta_k=\min_{1\leq i\leq m}\beta_{ik}$. 
In addition, suppose that 
$\BFS\cap\Cart_{\!\!i=1}^{\!\!m}\IDD_i\neq\emp$ 
and that either $(\forall i\in\{1,\ldots,m\})$ $f_i$ 
is cofinite or $(\forall i\in\{1,\ldots,m\})$ $\varphi_i$ 
is cofinite. Let $\varepsilon\in\big]0,1/\big(1+\sum_{k=1}^p
\sigma_k\beta_k^{-1}\big)\big[$, 
let $(\eta_n)_{n\in\NN}\in\ell_{+}^1(\NN)$, 
and let $(\gamma_n)_{n\in\NN}$ be a sequence in 
$\mathbb{R}$ such that 
\begin{equation}
\label{IIIe:46}
(\forall n\in\NN)\quad \varepsilon\leq\gamma_n\leq
\dfrac{1-\varepsilon}{\sum_{k=1}^p\sigma_k\beta_k^{-1}}
\quad\text{and}\quad(1+\eta_n)\gamma_n-\gamma_{n+1}\leq
\dfrac{\eta_n}{\sum_{k=1}^p\sigma_k\beta_k^{-1}}.
\end{equation}
Furthermore, let $(x_{i,0})_{1\leq i\leq m}\in
\Cart_{\!\!i=1}^{\!\!m}\IDD_i$ and iterate
\begin{equation}
\label{IIIe:63}
\begin{array}{l}
\text{for}\;n=0,1,\ldots\\
\left\lfloor
\begin{array}{l}
\text{for}\;i=1,\ldots, m\\
\left\lfloor
\begin{array}{l}
x_{i,n+1}=\Prox_{\gamma_n\varphi_i}^{f_i}\big(\nabla
f_i(x_{i,n})-\gamma_n\sum_{k=1}^pL_{ik}^*\nabla \psi_k
\big(\sum_{j=1}^mL_{jk}x_{j,n}\big)\big).
\end{array}
\right.\\
\end{array}
\right.\\
\end{array}
\end{equation}
Then there exists $(\overline{x}_i)_{1\leq i\leq m}\in\BFS$ 
such that the following hold.
\begin{enumerate}
\item
\label{IIIpp:2i} 
Suppose that $\BFS\cap\Cart_{\!\!i=1}^{\!\!m}\overline{\dom}f_i$ 
is a singleton. Then $(\forall i\in\{1,\ldots,m\})$ 
$x_{i,n}\weakly\overline{x}_i$.
\item
\label{IIIpp:2ii}  
For every $i\in\{1,\ldots,m\}$ and 
every $k\in\{1,\ldots,p\}$, suppose that 
$\nabla f_i$ and $\nabla\psi_k$ are weakly 
sequentially continuous, and that one 
of the following holds.
\begin{enumerate}
\item
\label{IIIpp:2iia}
$\dom\varphi_i\subset\IDD_i$.
\item
\label{IIIpp:2iib} 
$\dom f_i^*$ is open and $\nabla f_i^*$ 
is weakly sequentially continuous.
\end{enumerate}
Then $(\forall i\in\{1,\ldots,m\})$ 
$x_{i,n}\weakly\overline{x}_i$.
\end{enumerate}
\end{proposition}
\begin{proof}
Denote by $\XX$ and $\YY$ the standard vector product spaces 
$\Cart_{\!\!i=1}^{\!\!m}\XX_i$ and 
$\Cart_{\!\!k=1}^{\!\!p}\YY_k$ equipped with the norms 
$x=(x_i)_{1\leq i\leq m}\mapsto\sqrt{\sum_{i=1}^m\|x_i\|^2}$ and 
$y=(y_k)_{1\leq k\leq p}\mapsto\sqrt{\sum_{k=1}^p\|y_k\|^2}$,
respectively. Then $\XX^*$ is the vector product space 
$\Cart_{\!\!i=1}^{\!\!m}\XX_i^*$ equipped with the norm 
$x^*\mapsto\sqrt{\sum_{i=1}^m\|x_i^*\|^2}$ 
and $\YY^*$ is the vector product space 
$\Cart_{\!\!k=1}^{\!\!p}\YY_k^*$ equipped with the norm 
$y^*\mapsto\sqrt{\sum_{k=1}^p\|y_k^*\|^2}$. 
Let us introduce the functions and operator
\begin{equation}
\label{IIIe:48}
\begin{cases}
\varphi\colon\XX\to\RX\colon 
x\mapsto\sum_{i=1}^m\varphi_i(x_i)\\ 
f\colon\XX\to\RX\colon 
x\mapsto\sum_{i=1}^mf_i(x_i)\\
\psi\colon\YY\mapsto\RX\colon
y\mapsto\sum_{k=1}^p\psi_k(y_k)\\ 
L\colon\XX\to\YY\colon
x\mapsto
\big(\sum_{i=1}^mL_{ik}x_i\big)_{1\leq k\leq p}.
\end{cases}
\end{equation}
Then $\psi$ is G\^ateaux differentiable 
on $\IDDPS=\Cart_{\!\!k=1}^{\!\!p}\IDDPS_k$ and 
Problem~\ref{IIIpb:2} is a special case of Problem~\ref{IIIpb:1}. 
Since \eqref{IIIe:48} yields 
$\dom f^*=\Cart_{\!\!i=1}^{\!\!m}\dom f_i^*$ 
and $\dom\varphi^*=\Cart_{\!\!i=1}^{\!\!m}\dom\varphi_i^*$, 
we deduce from our assumptions that 
either $f$ is cofinite or $\varphi$ 
is cofinite. As in \eqref{IIIe:leg} 
and \eqref{IIIe:inter}, $f$ is a Legendre function 
and $\dom\varphi\cap\IDD\neq\emp$. 
In addition,
\begin{equation}
L\big(\IDD\big)=\underset{k=1}{\overset{p}{\Cart}}
\sum_{i=1}^mL_{ki}(\IDD_i)\subset
\underset{k=1}{\overset{p}{\Cart}}\IDDPS_k=\IDDPS.
\end{equation}
Now let $x\in\IDD$. First, to show that $D^f(x,\cdot)$ 
is coercive, we fix $\rho\in\RR$. On the one hand,
\begin{equation}
\label{IIIe:55}
\menge{z=(z_i)_{1\leq i\leq m}\in\XX}
{D^f(x,z)\leq\rho}
\subset\underset{i=1}{\overset{m}{\Cart}}\menge{z_i\in\XX_i}
{D^{f_i}(x_i,z_i)\leq\rho}.
\end{equation}
On the other hand, for every $i\in\{1,\ldots,m\}$, 
since $D^{f_i}(x_i,\cdot)$ is coercive, we deduce that 
\begin{equation}
\menge{z_i\in\XX_i}{D^{f_i}(x_i,z_i)\leq\rho}\quad
\text{is bounded.}
\end{equation}
Hence \eqref{IIIe:55} implies that 
$\menge{z\in\XX}{D^f(x,z)\leq\rho}$ is bounded 
and $D^f(x,\cdot)$ is therefore coercive. 
Next, set $\beta=1/\sum_{k=1}^p\sigma_k\beta_k^{-1}$. 
We shall show that $f\succcurlyeq\beta\psi\circ L$.
To this end, fix $z=(z_i)_{1\leq i\leq m}\in\IDD$. We have
\begin{align}
D^{\psi}(Lx,Lz)
&=\sum_{k=1}^pD^{\psi_k}
\Bigg(\sum_{i=1}^mL_{ik}x_i,\sum_{i=1}^mL_{ik}z_i\Bigg)
\nonumber\\
&\leq\sum_{k=1}^p\sum_{i=1}^m\sigma_k
D^{\psi_k}(L_{ik}x_i,L_{ik}z_i)\nonumber\\
&\leq\sum_{k=1}^p\sum_{i=1}^m\sigma_k\beta_{ik}^{-1}
D^{f_i}(x_i,z_i)\nonumber\\
&\leq\sum_{k=1}^p\sigma_k\beta_k^{-1}D^f(x,z).
\end{align}
Now let us set $(\forall n\in\NN)$ 
$x_n=(x_{i,n})_{1\leq i\leq m}$. 
By virtue of Proposition~\ref{IIIl:fproxprod}, 
\eqref{IIIe:63} is a particular case of \eqref{IIIal:2}. 

\ref{IIIpp:2i}:
Since $\BFS\cap\overline{\dom}f$ 
is a singleton, the claim follows from 
Theorem~\ref{IIIt:2}\ref{IIIt:2ia}.

\ref{IIIpp:2ii}:
Our assumptions on $(f_i)_{1\leq i\leq m}$ 
and $(\psi_k)_{1\leq k\leq p}$ 
imply that $\nabla f$ and $\nabla\psi$ 
are weakly sequentially continuous. 

\ref{IIIpp:2iia}:
Since $\BFS\subset\Cart_{\!\!i=1}^{\!\!m}\dom\varphi_i
\subset\Cart_{\!\!i=1}^{\!\!m}\IDD_i
=\IDD$, the claim follows from 
Theorem~\ref{IIIt:2}\ref{IIIt:2ib}.

\ref{IIIpp:2iib}:
Since, for every $i\in\{1,\ldots,m\}$, 
$\dom f_i^*$ is open and $\nabla f_i^*$ is weakly 
sequentially continuous, we deduce that 
$\dom f^*$ is open and $\nabla f^*$ 
is weakly sequentially continuous. The assertion 
therefore follows from Theorem~\ref{IIIt:2}\ref{IIIt:2ic}.
\end{proof}

\begin{example}
In Problem~\ref{IIIpb:2}, suppose that $m=1$, that $\XX_1$ and 
$(\YY_k)_{1\leq k\leq p}$ are Hilbert spaces, and 
that, for every $k\in\{1,\ldots,p\}$, 
$\varphi_k=\omega_k\|\cdot-r_k\|^2/2$, 
where $(\omega_k)_{1\leq k\leq p}\in\RPP^p$ 
and let $(r_k)_{1\leq k\leq p}\in\Cart_{\!\!k=1}^{\!\!p}\YY_k$. 
Then the weak convergence result in 
\cite[Proposition~6.3]{IIICV13b} without errors 
is a particular instance of Proposition~\ref{IIIpp:2} with 
$f_1=\|\cdot\|^2/2$.
\end{example}

\begin{example}
\label{IIIex:3}
Let $m$ and $p$ be strictly positive integers. For every 
$i\in\{1,\ldots,m\}$ and every $k\in\{1,\ldots,p\}$, let 
$\omega_{ik}\in\RPP$, let $\varrho_k\in\RPP$, and let 
$\varphi_i\in\Gamma_0(\mathbb{R})$ be cofinite. The problem is to
\begin{equation}
\label{IIIe:62}
\minimize{(\xi_1,\ldots,\xi_m)\in\RPP^m}{\sum_{i=1}^m
\varphi_i(\xi_i)+\sum_{k=1}^p
\left(-\ln\dfrac{\sum_{i=1}^m\omega_{ik}\xi_i}{\varrho_k}
+\dfrac{\sum_{i=1}^m\omega_{ik}\xi_i}{\varrho_k}-1\right)}.
\end{equation}
Denote by $\BFS$ the set of solutions to \eqref{IIIe:62} 
and suppose that $\BFS\cap\RPP^m\neq\emp$. 
Let
\begin{equation}
\vartheta\colon\RR\to\RX\colon\xi\mapsto
\begin{cases}
-\ln\xi,&\text{if}\quad\xi>0;\\
+\infty,&\text{otherwise}
\end{cases}
\end{equation}
be Burg entropy, let $\varepsilon\in\left]0,1/(1+p)\right[$, 
let $(\eta_n)_{n\in\NN}\in\ell_{+}^1(\mathbb{N})$, and let 
$(\gamma_n)_{n\in\NN}$ be a sequence in $\mathbb{R}$ such that
\begin{equation}
(\forall n\in\NN)\quad \varepsilon\leq\gamma_n\leq
p^{-1}(1-\varepsilon)\quad\text{and}\quad
(1+\eta_n)\gamma_n-\gamma_{n+1}\leq p^{-1}\eta_n.
\end{equation}
Let $(\xi_{i,0})_{1\leq i\leq m}\in\RPP^m$ and iterate
\begin{equation}
\label{IIIe:63b}
\begin{array}{l}
\text{for}\;n=0,1,\ldots\\
\left\lfloor
\begin{array}{l}
\text{for}\;i=1,\ldots, m\\
\left\lfloor
\begin{array}{l}
\xi_{i,n+1}=\Prox_{\gamma_n\varphi_i}^{\vartheta}
\Bigg(\dfrac{-1}{\xi_{i,n}}
-\gamma_n\sum\limits_{k=1}^p\omega_{ik}
\Bigg(\dfrac{-1}{\sum_{j=1}^m\omega_{jk}\xi_{j,n}}+
\dfrac{1}{\varrho_k}\Bigg)\Bigg).
\end{array}
\right.\\
\end{array}
\right.\\
\end{array}
\end{equation}
Then there exists $(\overline{\xi}_i)_{1\leq i\leq m}\in\BFS$ 
such that $(\forall i\in\{1,\ldots,m\})$ 
$\xi_{i,n}\to\overline{\xi}_i$. 
\end{example}
\begin{proof}
For every $i\in\{1,\ldots,m\}$ and every $k\in\{1,\ldots,p\}$, 
let us set $\XX_i=\RR$, $\YY_k=\RR$, 
$\psi_k=D^{\vartheta}(\cdot,\varrho_k)$, 
and $L_{ik}\colon\xi_i\mapsto\omega_{ik}\xi_i$. 
Then \eqref{IIIe:62} is a particular case of 
\eqref{IIIe:44}. Since $\psi$ is not differentiable 
on $\RR^p$, the standard forward-backward algorithm 
is inapplicable. We show that the problem can be solved 
by using Proposition~\ref{IIIpp:2}. 
First, let $(\xi_i)_{1\leq i\leq m}$ 
and $(\eta_i)_{1\leq i\leq m}$ be in $\RPP^m$, and consider
\begin{equation}
\label{IIIe:65}
\phi\colon\mathbb{R}\to\RX\colon\xi\mapsto
\begin{cases}
-\ln\xi+\xi-1, &\text{if}\quad\xi\in\RPP;\\
+\infty,&\text{otherwise}.
\end{cases}
\end{equation}
We see that $\phi$ is convex and positive. Thus, 
\begin{equation}
\phi\Bigg(\dfrac{\sum_{i=1}^m\xi_i}{\sum_{i=1}^m\eta_i}\Bigg)
=\phi\Bigg(\sum_{i=1}^m\dfrac{\eta_i}{\sum_{j=1}^m\eta_j}
\dfrac{\xi_i}{\eta_i}\Bigg)
\leq\sum_{i=1}^m\dfrac{\eta_i}{\sum_{j=1}^m\eta_j}
\phi\Bigg(\dfrac{\xi_i}{\eta_i}\Bigg)
\leq\sum_{i=1}^m\phi\Bigg(\dfrac{\xi_i}{\eta_i}\Bigg),
\end{equation}
and hence,
\begin{equation}
-\ln\dfrac{\sum_{i=1}^m\xi_i}{\sum_{i=1}^m\eta_i}
+\dfrac{\sum_{i=1}^m\xi_i}{\sum_{i=1}^m\eta_i}-1 
\leq\sum_{i=1}^m\Bigg(-\ln\dfrac{\xi_i}{\eta_i}
+\dfrac{\xi_i}{\eta_i}-1\Bigg).
\end{equation}
In turn,
\begin{equation}
\label{IIIe:66+}
D^{\vartheta}\Bigg(\sum_{i=1}^m\xi_i,\sum_{i=1}^m\eta_i\Bigg)
\leq\sum_{i=1}^mD^{\vartheta}(\xi_i,\eta_i).
\end{equation}
This shows that \eqref{IIIe:45} is satisfied with 
$(\forall k\in\{1,\ldots,p\})$ $\sigma_k=1$. 
Next, let us set $(\forall i\in\{1,\ldots,m\})$ 
$f_i=\vartheta$. Fix $i\in\{1,\ldots,m\}$ 
and $k\in\{1,\ldots,p\}$, and let 
$\xi_i$ and $\eta_i$ be in $\RPP$. Then
\begin{equation}
D^{\psi_k}(L_{ik}\xi_i,L_{ik}\eta_i)
=D^{\vartheta}(\omega_{ik}\xi_i,\omega_{ik}\eta_i)
=D^{\vartheta}(\xi_i,\eta_i)=D^{f_i}(\xi_i,\eta_i),
\end{equation}
which implies that $f_i\succcurlyeq\psi_k\circ L_{ik}$. 
In addition, since $\dom f_i^*=\left]-\infty,0\right[$ 
is open, \cite[Lemma~7.3(ix)]{IIIBBC01} asserts that 
$D^{f_i}(\xi_i,\cdot)$ is coercive. 
We therefore deduce the convergence result from 
Proposition~\ref{IIIpp:2}\ref{IIIpp:2iib}.
\end{proof}

\begin{example}
\label{IIIex:4}
Let $m$ and $p$ be strictly positive integers. For every 
$i\in\{1,\ldots,m\}$ and every $k\in\{1,\ldots,p\}$, let 
$\omega_{ik}\in\RPP$, let $\varrho_k\in\RPP$, and let 
$\varphi_i\in\Gamma_0(\mathbb{R})$. The problem is to
\begin{equation}
\label{IIIe:67}
\minimize{(\xi_1,\ldots,\xi_m)\in\left[0,+\infty\right[^m}
{\sum_{i=1}^m\varphi_i(\xi_i)+ \sum_{k=1}^p
\left(\left(\sum_{i=1}^m\omega_{ik}\xi_i\right)
\ln\dfrac{\sum_{i=1}^m\omega_{ik}\xi_i}{\varrho_k}
-\sum_{i=1}^m\omega_{ik}\xi_i+\varrho_k\right)}.
\end{equation}
Denote by $\BFS$ the set of solutions to \eqref{IIIe:67} 
and suppose that $\BFS\cap\RPP^m\neq\emp$.
Let
\begin{equation}
\vartheta\colon\RR\to\RX\colon\xi\mapsto
\begin{cases}
\xi\ln\xi-\xi,&\text{if}\;\xi\in\RPP;\\
0,&\text{if}\;\xi=0;\\
+\infty,&\text{otherwise}
\end{cases}
\end{equation}
be Boltzmann-Shannon entropy, let 
$\beta=\max\limits_{1\leq k\leq p}
\max\limits_{1\leq i\leq m}\omega_{ik}$, 
let $\varepsilon\in\left]0,1/(1+\beta)\right[$, 
let $(\eta_n)_{n\in\NN}\in\ell_{+}^1(\mathbb{N})$, 
and let $(\gamma_n)_{n\in\NN}$ 
be a sequence in $\mathbb{R}$ such that
\begin{equation}
(\forall n\in\NN)\quad \varepsilon\leq\gamma_n
\leq(p\beta)^{-1}(1-\varepsilon)\quad\text{and}\quad
(1+\eta_n)\gamma_n-\gamma_{n+1}\leq(p\beta)^{-1}\eta_n.
\end{equation}
Let $(\xi_{i,0})_{1\leq i\leq m}\in\RPP^m$ and iterate
\begin{equation}
\label{IIIe:68}
\begin{array}{l}
\text{for}\;n=0,1,\ldots\\
\left\lfloor
\begin{array}{l}
\text{for}\;i=1,\ldots, m\\
\left\lfloor
\begin{array}{l}
\xi_{i,n+1}=\Prox^{\vartheta}_{\gamma_n\varphi_i}\Big(\ln\xi_{i,n}
-\gamma_n\sum\limits_{k=1}^p\omega_{ik}
\Big(\ln\Big(\sum_{j=1}^m\omega_{jk}\xi_{j,n}\Big)
-\ln\varrho_k\Big)\Big).
\end{array}
\right.\\
\end{array}
\right.\\
\end{array}
\end{equation}
Then there exists $(\overline{\xi}_i)_{1\leq i\leq m}\in\BFS$ 
such that $(\forall i\in\{1,\ldots,m\})$ 
$\xi_{i,n}\to\overline{\xi}_i$.
\end{example}
\begin{proof}
For every $i\in\{1,\ldots,m\}$ and every $k\in\{1,\ldots,p\}$, 
let us set $\XX_i=\RR$, $\YY_k=\RR$, 
$\psi_k=D^{\vartheta}(\cdot,\varrho_k)$, 
and $L_{ik}\colon\xi_i\mapsto\omega_{ik}\xi_i$. 
Then \eqref{IIIe:67} is a particular case of 
\eqref{IIIe:44}. We cannot apply the standard forward-backward 
algorithm here since $\psi$ is not differentiable 
on $\RR^p$. We shall verify the assumptions of Proposition~\ref{IIIpp:2}. 
First, let $(\xi_i)_{1\leq i\leq m}$ 
and $(\eta_i)_{1\leq i\leq m}$ be in $\RPP^m$. 
Since
\begin{equation}
\label{IIIe:65+}
\phi\colon\mathbb{R}\to\RX\colon\xi\mapsto
\begin{cases}
\xi\ln\xi, &\text{if}\quad\xi\in\RPP;\\
0,&\text{if}\quad\xi=0;\\
+\infty,&\text{otherwise}
\end{cases}
\end{equation}
is convex, we have
\begin{equation}
\phi\Bigg(\dfrac{\sum_{i=1}^m\xi_i}{\sum_{i=1}^m\eta_i}\Bigg)
=\phi\Bigg(\sum_{i=1}^m\dfrac{\eta_i}{\sum_{j=1}^m\eta_j}
\dfrac{\xi_i}{\eta_i}\Bigg)
\leq\sum_{i=1}^m\dfrac{\eta_i}{\sum_{j=1}^m\eta_j}
\phi\Bigg(\dfrac{\xi_i}{\eta_i}\Bigg),
\end{equation}
and hence,
\begin{equation}
\dfrac{\sum_{i=1}^m\xi_i}{\sum_{i=1}^m\eta_i}
\ln\dfrac{\sum_{i=1}^m\xi_i}{\sum_{i=1}^m\eta_i}
\leq\sum_{i=1}^m\dfrac{\eta_i}{\sum_{j=1}^m\eta_j}
\dfrac{\xi_i}{\eta_i}\ln\dfrac{\xi_i}{\eta_i}
=\dfrac{\sum_{i=1}^m\xi_i\ln\dfrac{\xi_i}{\eta_i}}
{\sum_{i=1}^m\eta_i}.
\end{equation}
In turn,
\begin{equation}
\Bigg(\sum_{i=1}^m\xi_i\Bigg)
\ln\dfrac{\sum_{i=1}^m\xi_i}{\sum_{i=1}^m\eta_i}
\leq\sum_{i=1}^m\xi_i\ln\dfrac{\xi_i}{\eta_i},
\end{equation}
which implies that
\begin{align}
\label{IIIe:66}
D^{\vartheta}
\Bigg(\sum_{i=1}^m\xi_i,\sum_{i=1}^m\eta_i\Bigg)
&=\Bigg(\sum_{i=1}^m\xi_i\Bigg)
\ln\dfrac{\sum_{i=1}^m\xi_i}{\sum_{i=1}^m\eta_i}
-\sum_{i=1}^m\xi_i+\sum_{i=1}^m\eta_i\nonumber\\
&\leq\sum_{i=1}^m\Bigg(\xi_i\ln\dfrac{\xi_i}{\eta_i}
-\xi_i+\eta_i\Bigg)\nonumber\\
&=\sum_{i=1}^mD^{\vartheta}(\xi_i,\eta_i).
\end{align}
This shows that \eqref{IIIe:45} is satisfied with 
$(\forall k\in\{1,\ldots,p\})$ $\sigma_k=1$. 
Next, let us set $(\forall i\in\{1,\ldots,m\})$ 
$f_i=\vartheta$. Fix $i\in\{1,\ldots,m\}$ 
and $k\in\{1,\ldots,p\}$, and let 
$\xi_i$ and $\eta_i$ be in $\RPP$. Then
\begin{equation}
D^{\psi_k}(L_{ik}\xi_i,L_{ik}\eta_i)
=D^{\vartheta}(\omega_{ik}\xi_i,\omega_{ik}\eta_i)
=\omega_{ik}D^{\vartheta}(\xi_i,\eta_i)
\leq\beta D^{\vartheta}(\xi_i,\eta_i),
\end{equation}
which implies that 
$f_i\succcurlyeq\beta^{-1}\psi_k\circ L_{ik}$. 
In addition, since $f_i$ is supercoercive, 
$f_i$ is cofinite and 
\cite[Lemma~7.3(viii)]{IIIBBC01} asserts that 
$D^{f_i}(\xi_i,\cdot)$ is coercive. 
Therefore, the claim follows from 
Proposition~\ref{IIIpp:2}\ref{IIIpp:2iib}.
\end{proof}

\begin{remark}
The Bregman distance associated with Burg entropy, 
i.e., the Itakura-Saito divergence, is used in linear regression 
\cite[Section~3]{IIIBa13}. 
The Bregman distance associated with Boltzmann-Shannon entropy, 
i.e., the Kullback-Leibler divergence, is used in information 
theory \cite[Section~3]{IIIBa13} and image processing 
\cite{IIIByr93}.
\end{remark}

\vskip 0.5cm
\noindent{\bf Acknowledgment.} 
I would like to thank my doctoral advisor Professor Patrick L. 
Combettes for bringing this problem to my attention and for 
helpful discussions.

\end{document}